\documentclass[preprint,numbers,addressatend]{imsart}

\usepackage{amsthm,amsmath,natbib}
\usepackage{amsfonts,amssymb,graphicx}
\RequirePackage[dvips]{hyperref}
\usepackage{color}

\RequirePackage{hypernat}

\arxiv{}

\startlocaldefs
\numberwithin{equation}{section}
\newtheorem{thm}{Theorem}[section]
\newtheorem{lem}[thm]{Lemma}
\newtheorem{defn}[thm]{Definition}
\newtheorem{rem}[thm]{Remark}

\newcommand{\sgn}{{\rm sgn} \mspace{1mu}}
\renewcommand{\labelenumi}{\alph{enumi})}

\def\Indicator{\mathop{\hskip0pt{1}}\nolimits}
\def\eqdef{\triangleq}

\usepackage{graphicx}
\usepackage{verbatim}

\endlocaldefs

\begin{document}

\begin{frontmatter}

\title{Implicit Renewal Theorem for Trees \\ with General Weights}
\runtitle{Implicit Renewal Theory on Trees}


\author{\fnms{Predrag R.} \snm{Jelenkovi\'c}\ead[label=e1]{predrag@ee.columbia.edu}}
\address{Department of Electrical Engineering  \\ Columbia University \\ New York, NY 10027 \\ \printead{e1}}
\affiliation{Columbia University}
\and \hspace{4pt}
\author{\fnms{Mariana} \snm{Olvera-Cravioto}\ead[label=e2]{molvera@ieor.columbia.edu}}
\address{Department of Industrial Engineering \\ and Operations Research \\ Columbia University \\ New York, NY 10027 \\ \printead{e2}}
\affiliation{Columbia University}

\runauthor{P.R. Jelenkovi\'c  and M. Olvera-Cravioto}

\begin{abstract}
Consider distributional fixed point equations of the form 
$$R \stackrel{\mathcal{D}}{=} f( C_i, R_i, 1\le i\le N),$$
where $f(\cdot)$ is a possibly random real valued function, $N  \in \{0, 1, 2, 3, \dots\} \cup \{\infty\}$,
$\{C_i\}_{i=1}^N$ are real valued random weights 
and $\{R_i\}_{i\geq 1}$ are iid copies of $R$, independent of $(N, C_1,\dots, C_N)$;
$\stackrel{\mathcal{D}}{=}$ represents equality in distribution.
Fixed point equations of this type are of utmost importance for solving many applied probability problems, ranging from the average case analysis of algorithms to statistical physics. 
We develop an Implicit Renewal Theorem that enables the characterization of the power tail behavior of the solutions $R$ to many equations of multiplicative nature that fall into this category. This result extends the prior work in \cite{Jel_Olv_11}, 
which assumed nonnegative weights $\{C_i\}$, to general real valued weights. We illustrate the developed theorem by deriving the power tail asymptotics of the solution $R$ to the linear equation $R \stackrel{\mathcal{D}}{=} \sum_{i=1}^N C_i R_i + Q$.
\end{abstract}




\end{frontmatter}

\section{Introduction}

Many applied probability problems, ranging from the average case analysis of algorithms to statistical physics,
reduce to distributional fixed point equations of the form 
\begin{equation} \label{eq:general} 
R \stackrel{\mathcal{D}}{=} f( C_i, R_i, 1\le i\le N),
\end{equation}
where $f(\cdot)$ is a possibly random real valued function, $N  \in \mathbb{N}$, $\mathbb{N} = \{0, 1, 2, 3, \dots\} \cup \{\infty\}$,
$\{C_i\}_{i=1}^N$ are real valued random weights 
and $\{R_i\}_{i\geq 1}$ are iid copies of $R$, independent of $(N, C_1,\dots, C_N)$.
For a recent survey of a variety of problems where these equations appear see \cite{Aldo_Band_05}.
The solutions to these types of equations can be recursively constructed on a weighted branching tree,
where $N$ represents the generic branching variable and the $\{C_i\}_{i=1}^N$ are the branching weights. 
For this reason, we also refer to \eqref{eq:general} as recursions on weighted branching trees. 

In this paper, we develop an Implicit Renewal Theorem, stated in Theorem~\ref{T.NewGoldie}, 
that enables the characterization of the power tail behavior of the solutions $R$ to many equations of 
multiplicative nature of the form in \eqref{eq:general}.  
This result extends the prior work in \cite{Jel_Olv_11}, which assumed nonnegative weights $\{C_i\}$, 
to general real valued weights. This work also fully generalizes the Implicit Renewal Theorem of 
Goldie (1991) \cite{Goldie_91}, which was derived for equations of the form
$R\stackrel{\mathcal{D}}{=}f(C,R)$ (equivalently $N\equiv 1$ in our case),
to recursions (fixed point equations) on trees. 
Note that even in the classical non-branching problem the proof of the mixed sign case is quite involved, see the proof of Case~2 on pp.~145-149 in \cite{Goldie_91}.
We provide here a streamlined matrix form derivation of Theorem 2.3 in \cite{Goldie_91} that seamlessly extends to trees. 
For completeness, we also derive the lattice version of our implicit renewal theorem in Theorem~\ref{T.NewGoldie_Lattice}. 
One of the key observations leading to Theorems~\ref{T.NewGoldie} and~\ref{T.NewGoldie_Lattice} is that an appropriately 
constructed measure on a weighted branching tree is a matrix renewal measure, see Lemma~\ref{L.RenewalMeasure} and equation \eqref{eq:RenewalEquation}. 

We illustrate the developed theorem by deriving the power tail 
asymptotics of the nonhomogeneous linear recursion 
\begin{equation} \label{eq:IntroLinear} 
R \stackrel{\mathcal{D}}{=} \sum_{i=1}^N C_i R_i + Q,
\end{equation}
where $N  \in \mathbb{N} \cup \{\infty\}$,
$\{C_i\}_{i=1}^N$ are real valued random weights, $Q$ is a real valued random variable with  $P(Q \not = 0) > 0$
and $\{R_i\}_{i\geq 1}$ are iid copies of $R$, independent of $(N, C_1,\dots, C_N)$.
This recursion appeared recently in the stochastic analysis of Google's PageRank algorithm, 
see \cite{Volk_Litv_08, Jel_Olv_10, Jel_Olv_11} and the references therein for the latest work in the area. 
These types of weighted recursions, also known as weighted branching processes \cite{Rosler_93}, 
are found in the probabilistic analysis of other algorithms as well \cite{Ros_Rus_01, Nei_Rus_04}, e.g. Quicksort algorithm \cite{Fill_Jan_01}, see \cite{Aldo_Band_05, Alsm_Mein_10b, Iksanov_04, Jel_Olv_11, Jel_Olv_10, Liu_00, Rosler_93, Ros_Rus_01} for additional references. In addition, equation \eqref{eq:IntroLinear} generalizes other well studied problems in the literature, e.g.: for $N \equiv 1$, it reduces to an autoregressive process of order one and for $C_i \equiv$ constant, $R$ represents the busy period of an M/G/1 queue (e.g. see \cite{Zwart_01}). In the context of Google's PageRank algorithm, $R$ represents the rank of a generic page,
$N$ is the number of neighbors of such a page, and the $\{C_i\}$ are the weights that determine the contribution 
of each neighboring page to the total rank $R$.
Here, we argue that if the pointer by neighbor $i$ represents a negative reference, then the weight $C_i$ 
of such a reference should be negative as well, i.e., negative references should not increase the rank of $R$.
Hence, in this paper, we allow the weights $\{C_i\}$ to be possibly negative. 

Note that the majority of the work in the rest of the paper goes into the application of the main theorem to 
the nonhomogeneous recursion in \eqref{eq:IntroLinear}. In this regard, in Section~\ref{S.LinearRec}, we first construct an explicit solution \eqref{eq:ExplicitConstr} to \eqref{eq:IntroLinear} on a weighted branching tree and then provide sufficient conditions for the finiteness of moments of this solution in Lemma~\ref{L.Moments_R}. In addition, under quite general conditions, it can be shown that this solution is unique under iterations, see Lemma~4.5 in \cite{Jel_Olv_11}. However, the fixed point equation \eqref{eq:IntroLinear} can have additional stable solutions, as it was recently discovered in \cite{Alsm_Mein_10b}; earlier work for the case when $\{C_i\}, Q$ are deterministic real-valued constants can be found in \cite{Alsm_Rosl_05}. Furthermore, it is worth noting that our moment estimates are explicit, see Lemma \ref{L.GeneralMoment}, which may be of independent interest. Then, the main result, which characterizes the power-tail behavior of $R$ is presented in Theorem~\ref{T.LinearRecursion}. In addition, for integer power exponent ($\alpha \in \{ 1, 2, 3, \dots\}$) the asymptotic tail behavior can be explicitly computed, see Corollary~4.9 in  \cite{Jel_Olv_11}. Furthermore, for non integer $\alpha$, Lemma \ref{L.Alpha_Moments} 
can be used to derive an explicit bound on the tail behavior of $R$.

Similarly as in \cite{Jel_Olv_11}, our technique could be potentially applied to study the tail asymptotics of the solution to the critical, $E\left[ \sum_{i=1}^N C_i \right] = 1$, homogeneous linear equation 
\begin{equation} \label{eq:IntroLinearHomog} 
R \stackrel{\mathcal{D}}{=} \sum_{i=1}^N C_i R_i,
\end{equation}
where $\{C_i\}_{i=1}^N$ is a real valued random vector with $N \in \mathbb{N} \cup \{\infty\}$ and $\{R_i\}_{i\geq 1}$ is a sequence of iid random variables independent of $(N, C_1,\dots, C_N)$ having the same distribution as $R$; note that \cite{Jel_Olv_11} considered the nonnegative $\{C_i\}_{i=1}^N$ case.
See \cite{Liu_00, Iksanov_04, Negadailov_10} and the references therein for prior work on the power tail asymptotics of the homogeneous linear recursion. 
For additional references on weighted branching processes and multiplicative cascades see \cite{Als_Big_Mei_10, Alsm_Kuhl_07, Liu_00, Liu_98} and the references therein. For earlier historical references see \cite{Kah_Pey_76, Biggins_77, Holl_Ligg_81, Durr_Ligg_83}. 
In the same fashion, one can also study many other possibly nonlinear distributional 
equations, e.g.,
\begin{equation} \label{eq:IntroMaximum}
R \stackrel{\mathcal{D}}{=} \left(\bigvee_{i=1}^N C_i R_i \right) \vee Q,\quad 
R \stackrel{\mathcal{D}}{=} \left(\bigvee_{i=1}^N C_i R_i \right) + Q, \quad R \stackrel{\mathcal{D}}{=} \left(\sum_{i=1}^N C_i R_i \right) \vee Q;
\end{equation}
see \cite{Jel_Olv_11} for additional details on how Theorem~\ref{T.NewGoldie} can be applied to these,
as well as other stochastic recursions.
The majority of the proofs are postponed to Section~\ref{S.Proofs}.

\section{Model description} \label{S.ModelDescription}

First we construct a random tree $\mathcal{T}$. We use the notation $\emptyset$ to denote the root node of $\mathcal{T}$, and $A_n$, $n \geq 0$, to denote the set of all individuals in the $n$th generation of $\mathcal{T}$, $A_0 = \{\emptyset\}$. Let $Z_n$ be the number of individuals in the $n$th generation, that is, $Z_n = |A_n|$, where $| \cdot |$ denotes the cardinality of a set; in particular, $Z_0 = 1$. 

Next, let $\mathbb{N}_+ = \{1, 2, 3, \dots\}$ be the set of positive integers and let $U = \bigcup_{k=0}^\infty (\mathbb{N}_+)^k$ be the set of all finite sequences ${\bf i} = (i_1, i_2, \dots, i_n) \in U$, where by convention $\mathbb{N}_+^0 = \{ \emptyset\}$ contains the null sequence $\emptyset$. To ease the exposition, for a sequence ${\bf i} = (i_1, i_2, \dots, i_k) \in U$ we write ${\bf i}|n = (i_1, i_2, \dots, i_n)$, provided $k \geq n$, and  ${\bf i}|0 = \emptyset$ to denote the index truncation at level $n$, $n \geq 0$. Also, for ${\bf i} \in A_1$ we simply use the notation ${\bf i} = i_1$, that is, without the parenthesis. Similarly, for ${\bf i} = (i_1, \dots, i_n)$ we will use $({\bf i}, j) = (i_1,\dots, i_n, j)$ to denote the index concatenation operation, if ${\bf i} = \emptyset$, then $({\bf i}, j) = j$. 

We iteratively construct the tree as follows. Let $N$ be the number of individuals born to the root node $\emptyset$, $N_\emptyset = N$, and let $\{N_{\bf i} \}_{{\bf i} \in U}$ be iid copies of $N$. Define now 
\begin{equation} \label{eq:AnDef}
A_1 = \{ i: 1 \leq i \leq N \}, \quad A_n = \{ (i_1, i_2, \dots, i_n): (i_1, \dots, i_{n-1}) \in A_{n-1}, 1 \leq i_n \leq N_{(i_1, \dots, i_{n-1})} \}.
\end{equation}
It follows that the number of individuals $Z_n = |A_n|$ in the $n$th generation, $n \geq 1$, satisfies the branching recursion 
$$Z_{n} = \sum_{{\bf i } \in A_{n-1}} N_{\bf i}.$$ 

\begin{center}
\begin{figure}[t]
\begin{picture}(430,160)(0,0)
\put(0,0){\includegraphics[scale = 0.8, bb = 0 510 500 700, clip]{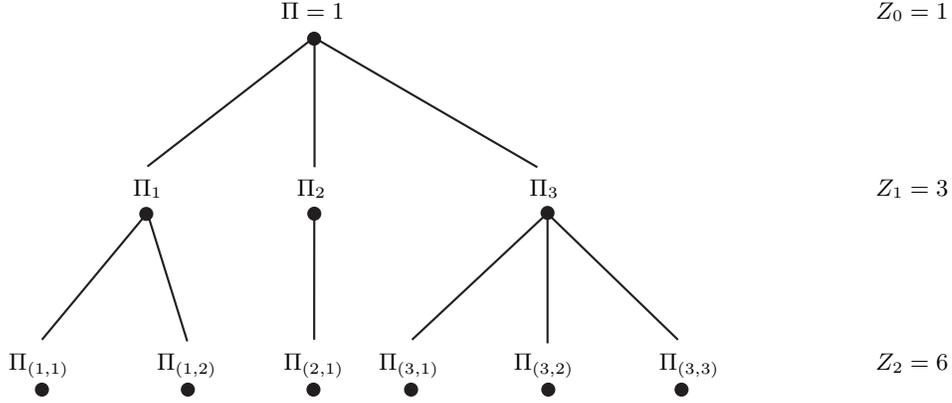}}
\put(125,150){\small $\Pi = 1$}
\put(69,83){\small $\Pi_{1}$}
\put(131,83){\small $\Pi_{2}$}
\put(219,83){\small $\Pi_{3}$}
\put(22,17){\small $\Pi_{(1,1)}$}
\put(78,17){\small $\Pi_{(1,2)}$}
\put(126,17){\small $\Pi_{(2,1)}$}
\put(162,17){\small $\Pi_{(3,1)}$}
\put(213,17){\small $\Pi_{(3,2)}$}
\put(268,17){\small $\Pi_{(3,3)}$}
\put(350,150){\small $Z_0 = 1$}
\put(350,83){\small $Z_1 = 3$}
\put(350,17){\small $Z_2 = 6$}
\end{picture}
\caption{Weighted branching tree}\label{F.Tree}
\end{figure}
\end{center}

Now, we construct the weighted branching tree $\mathcal{T}_{Q,C}$ as follows. Let $\{ (Q_{\bf i}, N_{\bf i}, C_{({\bf i}, 1)}, \dots, C_{({\bf i}, N_{\bf i})}) \}_{{\bf i} \in U}$ be a sequence of iid copies of $(Q, N, C_1, \dots, C_N)$. $N_\emptyset$ determines the number of nodes in the first generation of of $\mathcal{T}$ according to \eqref{eq:AnDef}, and each node in the first generation is then assigned its corresponding vector $(Q_i, N_i, C_{(i,1)}, \dots, C_{(i,N_i)})$ from the iid sequence defined above. In general, for $n \geq 2$, to each node ${\bf i} \in A_{n-1}$ we assign its corresponding $(Q_{\bf i}, N_{\bf i}, C_{({\bf i}, 1)}, \dots, C_{({\bf i}, N_{\bf i})})$ from the sequence and construct
$A_{n} = \{({\bf i}, i_{n}): {\bf i} \in A_{n-1}, 1 \leq i_{n} \leq N_{\bf i}\}$. 
For each node in $\mathcal{T}_{Q,C}$ we also define the weight $\Pi_{(i_1,\dots,i_n)}$ via the recursion
$$ \Pi_{i_1} =C_{i_1}, \qquad \Pi_{(i_1,\dots,i_n)} = C_{(i_1,\dots, i_n)} \Pi_{(i_1,\dots,i_{n-1})}, \quad n \geq 2,$$
where $\Pi =1$ is the weight of the root node. Note that the weight $\Pi_{(i_1,\dots, i_n)}$ is equal to the product of all the weights $C_{(\cdot)}$ along the branch leading to node $(i_1, \dots, i_n)$, as depicted in Figure \ref{F.Tree}.  
In some places, e.g. in the following section, the value of $Q$ may be of no importance, and thus we will consider a 
weighted branching tree defined by the smaller vector $(N, C_1, \dots, C_N)$.
This tree can be obtained form $\mathcal{T}_{Q,C}$ by simply disregarding the values for $Q_{(\cdot)}$ and is denoted by $\mathcal{T}_C$.

Studying recursions and fixed point equations embedded in this weighted branching tree is the objective of this paper.

\section{Implicit renewal theorem on trees} \label{S.Renewal}

In this section we present an extension of Goldie's Implicit Renewal Theorem \cite{Goldie_91} to weighted branching trees with real valued weights $\{C_i\}$. The key observation that facilitates this generalization is the following lemma which shows that a certain measure on a tree is a matrix product measure; its proof is given in Section~\ref{SS.ProofsRenewal}.  For the case of positive weights, a similar observation was made for a scalar measure in \cite{Biggins_Kyprianou_77}. Throughout the paper we use the standard convention $0^\alpha \log 0 = 0$ for all $\alpha > 0$.

Let ${\bf F} = (F_{ij} )$ be an $n\times n$ matrix whose elements are finite nonnegative measures 
concentrated on $\mathbb{R}$. The convolution ${\bf F} * {\bf G}$ of two such matrices is the matrix with 
elements $({\bf F} * {\bf G})_{ij} \triangleq \sum_{k=1}^n F_{ik} * G_{kj}$, $j = 1, \dots, n$, where $F_{ik}*G_{kj}$ is the convolution of individual measures. 

\begin{defn}
A matrix renewal measure is the matrix of measures
$${\bf U} = \sum_{k=0}^\infty {\bf F}^{*k},$$
where ${\bf F}^{*1} = {\bf F}$, ${\bf F}^{*(k+1)} = {\bf F}^{*k} * {\bf F} = {\bf F} * {\bf F}^{*k}$, ${\bf F}^{*0} = \delta_0 {\bf I}$, $\delta_0$ is the point measure at 0, and ${\bf I}$ is the identity $n \times n$ matrix. 
\end{defn}

\begin{defn}
A distribution $F$ on $\mathbb{R}$ is said to be {\em lattice} if it is concentrated on a set that forms an arithmetic progression, that is, on a set of points of the form $a + j\lambda$, where $a \in \mathbb{R}$, $\lambda > 0$ are constant numbers and $j \in \mathbb{Z} = \{0, \pm 1, \pm 2, \dots\}$. The largest number $\lambda$ with this property is called the span of $F$. A distribution that is not lattice is said to be  nonlattice. 
\end{defn}

\begin{lem} \label{L.RenewalMeasure}
Let  $\mathcal{T}_{C}$ be the weighted branching tree defined by the vector $(N,C_1, \dots, C_N)$, where $N \in \mathbb{N} \cup \{\infty\}$ and the $\{C_i\}$ are real valued.  
For any $n \in \mathbb{N}$ and ${\bf i} \in A_n$, let $V_{{\bf i}} = \log { |\Pi_{\bf i}|}$ and $X_{\bf i} = \sgn (\Pi_{\bf i})$; $V_\emptyset \equiv 0$, $X_\emptyset \equiv 1$. For $\alpha > 0$ define the measures
\begin{align*}
\mu_n^{(+)}(dt) &= e^{\alpha t} E\left[  \sum_{{\bf i} \in A_{n}} 1(X_{\bf i} = 1, V_{\bf i} \in dt )   \right], \\
\mu_n^{(-)}(dt) &= e^{\alpha t} E\left[  \sum_{{\bf i} \in A_{n}} 1(X_{\bf i} = -1, V_{\bf i}  \in dt )   \right],
\end{align*}
for $n = 0, 1,2,\dots$, and let $\eta_\pm(dt) = \mu_1^{(\pm)}(dt)$.  Suppose that $E\left[ \sum_{i=1}^N |C_i|^\alpha \log |C_i|  \right] \geq 0$ and $E\left[ \sum_{i=1}^N |C_i|^\alpha \right] = 1$. Then, $(\eta_+ + \eta_-)(\cdot)$ is a probability measure on $\mathbb{R}$ that places no mass at $-\infty$, and has mean
$$\int_{-\infty}^\infty u\, \eta_+(du) + \int_{-\infty}^\infty u\, \eta_-(du) = E\left[ \sum_{j=1}^N |C_j|^\alpha \log |C_j|  \right] .$$
Furthermore, if we let $\boldsymbol{\mu}_n = (\mu_n^{(+)}, \mu_n^{(-)})$, ${\bf e} = (1, 0)$ and ${\bf H} =  \left( \begin{matrix} 
\eta_+ & \eta_- \\ \eta_- & \eta_+ \end{matrix} \right)$, then
\begin{equation} \label{eq:convolution}
\boldsymbol{\mu}_n = (\mu_n^{(+)}, \mu_n^{(-)}) =  (1, 0)  \left( \begin{matrix} 
\eta_+ & \eta_- \\ \eta_- & \eta_+ \end{matrix} \right)^{*n} = {\bf e}  {\bf H}^{*n},
\end{equation}
where ${\bf H}^{*n}$ denotes the $n$th matrix convolution of ${\bf H}$ with itself.
\end{lem}

We now present a generalization of Goldie's Implicit Renewal Theorem \cite{Goldie_91} that will enable the analysis of recursions on weighted branching trees. Note that except for the independence assumption, the random variable $R$ and the vector $(N, C_1, \dots, C_N)$ are arbitrary, and therefore the applicability of this theorem goes beyond the linear recursion that we study here.

\begin{thm} \label{T.NewGoldie}
Let $(N,C_1, \dots, C_N)$ be a random vector, where $N \in \mathbb{N} \cup \{\infty\}$ and the $\{C_i\}$ are real valued. 
Suppose that there exists $j \geq 1$ with $P(N\ge j,|C_j|>0)>0$ such that the measure $P(\log |C_j|\in du, |C_j| > 0, N\ge j)$ is nonlattice.
Assume further that $E\left[ \sum_{j=1}^N |C_j|^\alpha \log |C_j|  \right] > 0$, $E\left[ \sum_{j=1}^N |C_j|^\alpha \right] = 1$, $E\left[ \sum_{j=1}^N |C_j|^\gamma \right] < \infty$ for some $0 \leq \gamma < \alpha$, and that $R$ is independent of $(N, C_1, \dots, C_N)$. 
\begin{enumerate}
\item If $\{ C_i \} \geq 0$ a.s., $E[((R)^+)^\beta] < \infty$ for any $0< \beta < \alpha$, and 
\begin{equation} \label{eq:Goldie_condition1}
\int_0^\infty \left| P(R > t) - E\left[ \sum_{j=1}^N 1(C_j R > t) \right] \right| t^{\alpha-1} dt < \infty,
\end{equation}
or, respectively, $E[((R^-)^\beta] < \infty$ for any $0< \beta < \alpha$, and
\begin{equation} \label{eq:Goldie_condition2}
\int_0^\infty \left| P(R < -t) - E\left[ \sum_{j=1}^N 1(C_j R <- t ) \right] \right| t^{\alpha-1} dt < \infty,
\end{equation}
then
$$P(R > t) \sim H_+ t^{-\alpha}, \qquad t \to \infty,$$
or, respectively,
$$P(R < -t) \sim H_- t^{-\alpha}, \qquad t \to \infty,$$
where $0 \leq H_\pm < \infty$ are given by
\begin{align*}
H_\pm &= \frac{1}{E\left[ \sum_{j=1}^N |C_j|^\alpha \log |C_j|  \right] } \int_{0}^\infty v^{\alpha-1} \left( P((\pm 1) R > v) - E\left[ \sum_{j=1}^{N} 1((\pm 1) C_{j} R > v ) \right]    \right) dv .
\end{align*}
\item If $P(C_j < 0) > 0$ for some $j \geq 1$,  $E[|R|^\beta] < \infty$ for any $0< \beta < \alpha$, and both \eqref{eq:Goldie_condition1} and \eqref{eq:Goldie_condition2} are satisfied, then
$$P(R > t) \sim P(R < -t) \sim H t^{-\alpha}, \qquad t \to \infty,$$
where $0 \leq H = (H_+ + H_-)/2 < \infty$ is given by
\begin{align*}
H &= \frac{1}{2E\left[ \sum_{j=1}^N |C_j|^\alpha \log |C_j|  \right] } \int_{0}^\infty v^{\alpha-1} \left( P(|R| > v) - E\left[ \sum_{j=1}^{N} 1(|C_{j} R| > v ) \right]    \right) dv .
\end{align*}
\end{enumerate}
\end{thm}

\begin{rem}
(i) As pointed out in \cite{Goldie_91}, the statement of the theorem only has content when $R^+$, $R^-$ or $|R|$, respectively, has infinite moments of order $\alpha$, since otherwise $H_+$, $H_-$ or $H$, respectively, are zero. (ii) Note that the case of nonnegative weights $\{C_i\} \geq 0$ a.s. was recently proved in Theorem~3.2 in \cite{Jel_Olv_11}. Here, in the proof of Theorem~\ref{T.NewGoldie} we refer to it as Case a), and provide an alternative proof that does not require the finiteness of $E\left[ \sum_{j=1}^N |C_j|^\alpha \log|C_j| \right]$; when this expectation is infinite the constants $H_\pm, H$ are zero which can be interpreted as $R$ having lighter tails than $t^{-\alpha}$. (iii) We also point out that our proof provides a streamlined derivation of the classical theorem of Goldie \cite{Goldie_91} ($N = 1$) through the use of a matrix renewal measure. (iv) Note that in both cases, (a) and (b), provided that \eqref{eq:Goldie_condition1} and \eqref{eq:Goldie_condition2} hold, we have
$$P(|R| > t) \sim (H_++H_-) t^{-\alpha}, \qquad \text{as } t \to \infty.$$ (v) It appears, as noted in \cite{Goldie_91}, that some of the early ideas of applying renewal theory to study the power tail asymptotics of autoregressive processes (perpetuities) is due to \cite{Grincevicius_75}. 
\end{rem}

We give below the corresponding theorem for the lattice case.

\begin{thm} \label{T.NewGoldie_Lattice}
Let $(N,C_1, \dots, C_N)$ be a random vector, where $N \in \mathbb{N} \cup \{\infty\}$ and the $\{C_i\}$ are real valued random variables such that for all $i$, given $|C_i| > 0$,  $ \log |C_i|  \subseteq L$, where $L = \{ \lambda j: j \in \mathbb{Z}\}$ for some $\lambda > 0$.
Assume further that $E\left[ \sum_{j=1}^N |C_j|^\alpha \log |C_j|  \right] > 0$, $E\left[ \sum_{j=1}^N |C_j|^\alpha \right] = 1$, $E\left[ \sum_{j=1}^N |C_j|^\gamma \right] < \infty$ for some $0 \leq \gamma < \alpha$, and that $R$ is independent of $(N, C_1, \dots, C_N)$. 
\begin{enumerate}
\item If $\{ C_i \} \geq 0$ a.s., $E[((R)^+)^\beta] < \infty$ for any $0< \beta < \alpha$, and 
\begin{equation} \label{eq:Goldie_condition1_Discrete}
\int_0^\infty \left| P(R > t) - E\left[ \sum_{j=1}^N 1(C_j R > t ) \right] \right| t^{\alpha-1} dt < \infty,
\end{equation}
or, respectively, $E[((R^-)^\beta] < \infty$ for any $0< \beta < \alpha$, and
\begin{equation} \label{eq:Goldie_condition2_Discrete}
\int_0^\infty \left| P(R < -t) - E\left[ \sum_{j=1}^N 1(C_j R <- t ) \right] \right| t^{\alpha-1} dt < \infty,
\end{equation}
then, for almost every $t \in \mathbb{R}$ (with respect to the Lebesgue measure), 
$$P(R > e^{t+\lambda n}) \sim H_+(t)  e^{-\alpha(t+\lambda n)}, \qquad n \to \infty,$$
or, respectively,
$$P(R < -e^{t+\lambda n}) \sim H_-(t)  e^{-\alpha(t+\lambda n)}, \qquad n \to \infty,$$
where $0 \leq H_\pm(t) < \infty$ are given by
\begin{align*}
H_\pm(t) &= \frac{\lambda}{E\left[ \sum_{j=1}^N |C_j|^\alpha \log |C_j|  \right] }\sum_{k=-\infty}^\infty e^{\alpha(t+k\lambda)} \left( P((\pm 1)R > e^{t+k\lambda}) - E\left[ \sum_{j=1}^N 1((\pm 1)C_j R > e^{t+k\lambda} ) \right] \right)  .
\end{align*}
\item If $P(C_j < 0) > 0$ for some $j \geq 1$,  $E[|R|^\beta] < \infty$ for any $0< \beta < \alpha$, and both \eqref{eq:Goldie_condition1} and \eqref{eq:Goldie_condition2} are satisfied, then, for almost every $t \in \mathbb{R}$ (with respect to the Lebesgue measure),
$$P(R > e^{t+\lambda n}) \sim P(R < -e^{t+\lambda n}) \sim H(t) e^{-\alpha(t+\lambda n)}, \qquad n\to \infty,$$
where $0 \leq H(t) = (H_+(t) + H_-(t))/2 < \infty$ is given by
\begin{align*}
H(t) &= \frac{\lambda}{E\left[ \sum_{j=1}^N |C_j|^\alpha \log |C_j|  \right] }\sum_{k=-\infty}^\infty e^{\alpha(t+k\lambda)} \left( P(|R| > e^{t+k\lambda}) - E\left[ \sum_{j=1}^N 1( |C_j R| > e^{t+k\lambda}) \right] \right)   .
\end{align*}
\end{enumerate}
\end{thm}

\begin{rem}
(i) The absolute integrability conditions \eqref{eq:Goldie_condition1_Discrete} and \eqref{eq:Goldie_condition2_Discrete} can be replaced by
$$\sup_{0\leq t \leq \lambda} \sum_{k=-\infty}^\infty e^{\alpha(t+k\lambda)} \left| P((\pm 1)R > e^{t+k\lambda}) - E\left[ \sum_{j=1}^N 1((\pm 1)C_j R > e^{t + \lambda k}) \right] \right| < \infty.$$
(ii) This theorem can be used to derive the tail behavior of the solutions to a variety of 
fixed point equations under the lattice assumption, e.g., those studied in \cite{Jel_Olv_11} for the 
nonlattice case. In particular, one can obtain an alternative derivation of existing results in the 
literature for the homogeneous equation ($Q=0$) with nonnegative weights ($C_i\ge 0$) under the 
lattice assumption, e.g., see Proposition~7 in \cite{Iksanov_04} and Theorem~29(b) in \cite{Negadailov_10}.
 We refrain from such possible derivations here since our primary 
motivation for this work is the nonhomogeneous linear recursion \eqref{eq:IntroLinear}. In addition, we focus 
on the nonlattice assumption since the results tend to be more explicit.  (iii) Early results for perpetuities ($R \stackrel{\mathcal{D}}{=} CR+Q$) in the lattice case can be found in 
Theorem~2(b) of \cite{Grincevicius_75}.
\end{rem}

Before going into the proof of Theorem \ref{T.NewGoldie} we need the following monotone density lemma, which is taken from \cite{Jel_Olv_11}. Since the proof of the lattice case is very similar to that of Theorem \ref{T.NewGoldie}, we postpone the proof of Theorem \ref{T.NewGoldie_Lattice} to Section \ref{SS.ProofsRenewal}.

\begin{lem} \label{L.Derivative}
Let $\alpha, \beta > 0$ and $0 \leq H < \infty$. Suppose $\int_0^t v^{\alpha+\beta-1} P(R > v) dv \sim H t^{\beta}/\beta$ as $t \to \infty$. Then,
$$P(R > t) \sim H t^{-\alpha}, \qquad t \to \infty.$$
\end{lem}

\bigskip

\begin{proof}[Proof of Theorem \ref{T.NewGoldie}]
Let  $\mathcal{T}_{C}$ be the weighted branching tree defined by the vector $(N,C_1, \dots, C_N)$.  For each ${\bf i} \in A_n$ and all $k \leq n$ define $V_{{\bf i}|k} = \log |\Pi_{{\bf i}|k}|$; note that $\Pi_{{\bf i}|k}$ is independent of $N_{{\bf i}|k}$ but not of $N_{{\bf i}|s}$ for any $0\leq s \leq k-1$. Also note that ${\bf i}|n = {\bf i}$ since ${\bf i} \in A_n$. Let  $\mathcal{F}_{k}$, $k \geq 1$, denote the $\sigma$-algebra generated by $\{ (N_{\bf i}, C_{({\bf i},1)}, \dots, C_{({\bf i}, N_{\bf i})}): {\bf i} \in A_j, 0 \leq j \leq k-1\}$, and let $\mathcal{F}_0 = \sigma(\emptyset, \Omega)$, $\Pi_{{\bf i}|0} \equiv 1$. Assume also that $R$ is independent of the entire weighted tree, $\mathcal{T}_C$.  Then, for any $t \in \mathbb{R}$, we can write $P(R > e^t)$ via a telescoping sum as follows (note that all the expectations in \eqref{eq:telescoping} are finite by Markov's inequality and \eqref{eq:PiMoments})
\begin{align}
&P(R > e^t) \notag \\
&= \sum_{k=0}^{n-1} \left( E\left[ \sum_{({\bf i}|k) \in A_{k}} 1(\Pi_{{\bf i}|k} R > e^t ) \right] - E\left[ \sum_{({\bf i}|k+1) \in A_{k+1}} 1(\Pi_{{\bf i}|k+1} R > e^t) \right]  \right) \label{eq:telescoping} \\
&\hspace{5mm} + E\left[ \sum_{({\bf i}|n) \in A_n} 1(\Pi_{{\bf i}|n} R > e^t ) \right] \notag \\
&= \sum_{k=0}^{n-1}  E\left[ \sum_{({\bf i}|k) \in A_{k}} \left( 1(\Pi_{{\bf i}|k} R > e^t ) -  \sum_{j=1}^{N_{{\bf i}|k}} 1(\Pi_{{\bf i}|k} C_{({\bf i}|k,j)} R > e^t ) \right) \right]  + E\left[ \sum_{({\bf i}|n) \in A_n} 1(\Pi_{{\bf i}|n} R > e^t ) \right] \notag \\
&= \sum_{k=0}^{n-1} E\left[ \sum_{({\bf i}|k) \in A_{k}} \left( 1(X_{{\bf i}|k} = 1,   R > e^{t-V_{{\bf i}|k}} ) -  \sum_{j=1}^{N_{{\bf i}|k}} 1( X_{{\bf i}|k} = 1,  C_{({\bf i}|k,j)} R > e^{t-V_{{\bf i}|k}} ) \right) \right]  \notag \\
&\hspace{5mm} +  \sum_{k=0}^{n-1} E\left[ \sum_{({\bf i}|k) \in A_{k}} \left( 1(X_{{\bf i}|k} = -1,  R <- e^{t-V_{{\bf i}|k}} ) -  \sum_{j=1}^{N_{{\bf i}|k}} 1( X_{{\bf i}|k}= -1,  C_{({\bf i}|k,j)} R < - e^{t-V_{{\bf i}|k}} ) \right) \right]  \notag \\
&\hspace{5mm} + E\left[ \sum_{({\bf i}|n) \in A_n} 1(\Pi_{{\bf i}|n} R > e^t ) \right] \notag \\
&= \sum_{k=0}^{n-1} E\left[ \sum_{({\bf i}|k) \in A_{k}} 1(X_{{\bf i}|k} = 1) E \left[ \left.   1( R > e^{t-V_{{\bf i}|k}} ) -  \sum_{j=1}^{N_{{\bf i}|k}} 1( C_{({\bf i}|k,j)} R > e^{t-V_{{\bf i}|k}} ) \right| \mathcal{F}_k \right] \right]  \notag \\
&\hspace{5mm} +  \sum_{k=0}^{n-1} E\left[ \sum_{({\bf i}|k) \in A_{k}}  1(X_{{\bf i}|k} = -1) E \left[ \left. 1( R <- e^{t-V_{{\bf i}|k}} ) -  \sum_{j=1}^{N_{{\bf i}|k}} 1( C_{({\bf i}|k,j)} R < - e^{t-V_{{\bf i}|k}} )   \right| \mathcal{F}_k \right] \right]  \notag \\
&\hspace{5mm} + E\left[ \sum_{({\bf i}|n) \in A_n} 1(\Pi_{{\bf i}|n} R > e^t ) \right].
\end{align}

Now, define the measures $\mu_n^{(+)}$ and $\mu_n^{(-)}$ according to Lemma \ref{L.RenewalMeasure} and let
$$\nu_n^{(+)}(dt) = \sum_{k=0}^n \mu_k^{(+)}(dt), \qquad g_+(t) = e^{\alpha t} \left( P(R > e^t) - E\left[ \sum_{j=1}^{N} 1(C_{j} R > e^{t}) \right] \right),$$
$$\nu_n^{(-)}(dt) = \sum_{k=0}^n \mu_k^{(-)}(dt) , \qquad g_-(t) = e^{\alpha t} \left( P(R <- e^t) - E\left[ \sum_{j=1}^{N} 1(C_{j} R <- e^{t} ) \right]    \right),$$
$$r(t) = e^{\alpha t} P(R > e^t) \qquad \text{and} \qquad \delta_n(t) = e^{\alpha t} E\left[ \sum_{({\bf i}|n) \in A_n} 1(\Pi_{{\bf i}|n} R > e^t ) \right].$$

Since $R$ and $(N_{{\bf i} | k}, C_{({\bf i}, 1)}, \dots, C_{({\bf i}, N_{\bf i})})$ are independent of $\mathcal{F}_k$, then
\begin{align*}
&E \left[ \left.   1( R > e^{t-V_{{\bf i}|k}} ) -  \sum_{j=1}^{N_{{\bf i}|k}} 1( C_{({\bf i}|k,j)} R > e^{t-V_{{\bf i}|k}} ) \right| \mathcal{F}_k \right] = e^{\alpha (V_{{\bf i}|k} - t)} g_+(t - V_{{\bf i}|k}), \quad \text{and} \\
&E \left[ \left. 1( R <- e^{t-V_{{\bf i}|k}} ) -  \sum_{j=1}^{N_{{\bf i}|k}} 1( C_{({\bf i}|k,j)} R < - e^{t-V_{{\bf i}|k}} )   \right| \mathcal{F}_k \right] = e^{\alpha (V_{{\bf i}|k} - t)} g_-(t - V_{{\bf i}|k}).
\end{align*}
It follows that for any $t \in \mathbb{R}$ and $n \in \mathbb{N}$, 
$$r(t) = (g_+*\nu_{n-1}^{(+)})(t) + (g_-*\nu_{n-1}^{(-)})(t)  +  \delta_n(t).$$
Next, for any $\beta > 0$, define the operator
$$\breve{f}(t) = \int_{-\infty}^t e^{-\beta(t-u)} f(u) \, du$$
and note that
\begin{align}
\breve{r}(t) &= \int_{-\infty}^t e^{-\beta(t-u)} (g_+*\nu_{n-1}^{(+)})(u) \, du +  \int_{-\infty}^t e^{-\beta(t-u)} (g_-*\nu_{n-1}^{(-)})(u) \, du  +  \breve{\delta}_n (t) \notag \\
&= \int_{-\infty}^t e^{-\beta(t-u)} \int_{-\infty}^\infty g_+(u-v) \nu_{n-1}^{(+)}(dv) \, du + \int_{-\infty}^t e^{-\beta(t-u)} \int_{-\infty}^\infty g_-(u-v) \nu_{n-1}^{(-)}(dv) \, du + \breve{\delta}_n (t) \notag \\
&= \int_{-\infty}^\infty \int_{-\infty}^t e^{-\beta(t-u)} g_+(u-v) \, du \, \nu_{n-1}^{(+)}(dv) +  \int_{-\infty}^\infty \int_{-\infty}^t e^{-\beta(t-u)} g_-(u-v) \, du \, \nu_{n-1}^{(-)}(dv) + \breve{\delta}_n (t) \notag \\
&= \int_{-\infty}^\infty \breve{g}_+(t-v) \, \nu_{n-1}^{(+)}(dv) + \int_{-\infty}^\infty \breve{g}_-(t-v) \, \nu_{n-1}^{(-)}(dv) + \breve{\delta}_n (t) \notag \\
&= (\breve{g}_+* \nu_{n-1}^{(+)})(t) +  (\breve{g}_-* \nu_{n-1}^{(-)})(t) + \breve{\delta}_n(t) . \label{eq:SmoothOperator}
\end{align}

Now, we will show that one can pass $n \to \infty$ in the preceding identity. To this end, let $\eta_\pm(du) = \mu_1^{(\pm)}(du)$, and note that by Lemma \ref{L.RenewalMeasure} $(\eta_++\eta_-)(\cdot)$ is a probability measure on $\mathbb{R}$ that places no mass at $-\infty$ and has mean,
$$\mu \triangleq \int_{-\infty}^\infty u\, \eta_+(du) + \int_{-\infty}^\infty u\, \eta_-(du) = E\left[ \sum_{j=1}^N |C_j|^\alpha \log |C_j|  \right] > 0 .$$
To see that $(\eta_+ +  \eta_-)(\cdot)$ is nonlattice note that by assumption the measure $P(\log |C_j| \in du, |C_j| > 0, N\ge j)$ is nonlattice, since, if we suppose to the contrary that it is lattice on a lattice set $L$, then on the complement $L^c$ of this set we have (by conditioning on $N$)
$$0=E\left[ \sum_{i=1}^N 1(\log |C_i| \in L^c, |C_i|>0 )  \right]\ge P(\log |C_j| \in L^c, |C_j|>0,N \ge j)>0,$$ 
which is a contradiction.  

Moreover, in the notation of Lemma \ref{L.RenewalMeasure}, $\boldsymbol{\mu}_k = (\mu_k^{(+)}, \mu_k^{(-)})$, ${\bf e} = (1, 0)$ and ${\bf H} =  \left( \begin{matrix} 
\eta_+ & \eta_- \\ \eta_- & \eta_+ \end{matrix} \right)$, which gives
\begin{equation} \label{eq:RenewalMeasure}
\boldsymbol{\nu} = \left( \nu^{(+)} , \nu^{(-)}  \right) \triangleq  \sum_{k=0}^\infty \left(  \mu_k^{(+)} ,  \mu_k^{(-)}  \right)   = \sum_{k=0}^\infty \boldsymbol{\mu}_k = \sum_{k=0}^\infty {\bf e} {\bf H}^{*k} = {\bf e} \sum_{k=0}^\infty {\bf H}^{*k}. 
\end{equation}
Also, $\eta_+ + \eta_-$ being nonlattice implies that at least one of $\eta_+$ or $\eta_-$ is nonlattice, and therefore ${\bf H}$ is nonlattice.

Since $\mu \neq 0$, then $(|f|*\nu^{(\pm)})(t) < \infty$ for all $t$ whenever $f$ is directly Riemann integrable. By \eqref{eq:Goldie_condition1} and \eqref{eq:Goldie_condition2} we know that $g_{\pm} \in L^1$, so by Lemma 9.1 from \cite{Goldie_91}, $\breve{g}_{\pm}$ is directly Riemann integrable, resulting in $(|\breve{g}_{\pm}|*\nu^{(\pm)})(t) < \infty$ for all $t$.  Thus, 
$(|\breve{g}_\pm|*\nu^{(\pm)})(t) = E\left[ \sum_{k=0}^\infty  \sum_{({\bf i}|k) \in A_{k}} e^{\alpha V_{{\bf i}|k}} |\breve{g}_\pm(t-V_{{\bf i}|k})| \Indicator(X_{{\bf i}|k} = \pm 1) \right]  < \infty$, implying that $E\left[ \sum_{k=0}^\infty  \sum_{({\bf i}|k) \in A_{k}} e^{\alpha V_{{\bf i}|k}} \breve{g}_\pm(t-V_{{\bf i}|k}) \Indicator (X_{{\bf i}|k} = \pm 1) \right] $ exist, and by Fubini's theorem,
\begin{align*}
(\breve{g}_\pm*\nu^{(\pm)})(t) &= E\left[ \sum_{k=0}^\infty  \sum_{({\bf i}|k) \in A_{k}} e^{\alpha V_{{\bf i}|k}} \breve{g}_\pm(t-V_{{\bf i}|k}) \Indicator(X_{{\bf i}|k} = \pm 1) \right]  \\
&= \sum_{k=0}^\infty E\left[ \sum_{({\bf i}|k) \in A_{k}} e^{\alpha V_{{\bf i}|k}} \breve{g}_\pm(t-V_{{\bf i}|k}) \Indicator(X_{{\bf i}|k} = \pm 1) \right]  = \lim_{n\to \infty}  (\breve{g}_\pm*\nu_n^{(\pm)})(t).
\end{align*}

For case b), to see that $\breve{\delta}_n(t) \to 0$ as $n \to \infty$ for all fixed $t$, note that from the assumptions $E\left[ \sum_{j=1}^N |C_j|^\alpha \right] = 1$,  $E\left[ \sum_{j=1}^N |C_j|^\alpha \log |C_j| \right] > 0$, , and $E\left[ \sum_{j=1}^N |C_j|^\gamma \right] < \infty$ for some $0 \leq \gamma < \alpha$, there exists
$0 < \beta <\alpha$ such that $E\left[ \sum_{j=1}^N |C_j|^{\beta} \right] < 1$ (by convexity). Therefore, for such $\beta$, 
\begin{align}
\breve{\delta}_n(t) &= \int_{-\infty}^t e^{-\beta(t-u)} e^{\alpha u} E\left[ \sum_{({\bf i}|n) \in A_n} 1(\Pi_{{\bf i}|n} R > e^{u}) \right]  du \notag \\
&\leq e^{(\alpha-\beta)t} E\left[ \sum_{({\bf i}|n) \in A_n}  \int_{-\infty}^t e^{\beta u} 1(|\Pi_{{\bf i}|n} R| > e^{u} ) du \right] \notag   \\
&=  e^{(\alpha-\beta)t} E\left[ \sum_{({\bf i}|n) \in A_n} \int_0^{\min\{t, \log (|\Pi_{{\bf i}|n} R|)  \} } e^{\beta u}   du \right] \notag \\
&\leq \frac{e^{(\alpha-\beta)t}}{\beta} E\left[ \sum_{({\bf i}|n) \in A_n}  |\Pi_{{\bf i}|n} R|^\beta \right] . \label{eq:Geometric}
\end{align}
Similarly, one obtains bounds for case a) by replacing $|R|$ by either $R^+$ or $R^-$. 

It remains to show that the expectation in \eqref{eq:Geometric} converges to zero as $n \to \infty$. First note that from the independence of $R$ and $\mathcal{T}_C$, 
$$E\left[ \sum_{({\bf i}|n) \in A_n}  |\Pi_{{\bf i}|n} R|^\beta \right]  = E[|R|^\beta] E\left[ \sum_{({\bf i}|n) \in A_n}  |\Pi_{{\bf i}|n}|^\beta \right],$$
where $E[|R|^\beta] < \infty$, for $0 < \beta < \alpha$. For the expectation involving $\Pi_{{\bf i}|n}$ condition on $\mathcal{F}_{n-1}$ and use the independence of $(N_{{\bf i}|n-1}, C_{({\bf i}|n-1,1)}, \dots, C_{({\bf i}|n-1, N_{{\bf i}|n-1})})$ from $\mathcal{F}_{n-1}$ as follows
\begin{align}
E\left[   \sum_{({\bf i}|n) \in A_n}  |\Pi_{{\bf i}|n}|^{\beta} \right] &= E\left[   \sum_{({\bf i}|n-1) \in A_{n-1}}  E\left[ \left. \sum_{j=1}^{N_{{\bf i}|n-1}}  |\Pi_{{\bf i}|n-1}|^{\beta} |C_{({\bf i}|n-1, j)}|^\beta \right| \mathcal{F}_{n-1} \right]  \right] \notag \\
&= E\left[   \sum_{({\bf i}|n-1) \in A_{n-1}} |\Pi_{{\bf i}|n-1}|^\beta  E\left[ \left. \sum_{j=1}^{N_{{\bf i}|n-1}} |C_{({\bf i}|n-1, j)}|^\beta \right| \mathcal{F}_{n-1} \right]  \right]  \notag \\
&= E\left[ \sum_{j=1}^N |C_j|^\beta \right] E\left[   \sum_{({\bf i}|n-1) \in A_{n-1}} |\Pi_{{\bf i}|n-1}|^\beta    \right]  \notag \\
&= \left( E\left[     \sum_{j=1}^{N}  |C_{j}|^{\beta}   \right] \right)^n \qquad \text{(iterating $n-1$ times)}. \label{eq:PiMoments}
\end{align}
Since $E\left[     \sum_{j=1}^{N}  |C_{j}|^{\beta}   \right] < 1$, then the above converges to zero as $n \to \infty$. Hence, the preceding arguments allow us to pass $n \to \infty$ in \eqref{eq:SmoothOperator}, and obtain
\begin{equation}
\label{eq:RenewalEquation}
\breve{r}(t) = (\boldsymbol{\nu}*{\bf g})(t) = {\bf e} \left( {\bf U} * {\bf g} \right) (t),
\end{equation}
where ${\bf g} = (\breve{g}_+ , \breve{g}_-)^{\sc T}$ and ${\bf U} = \sum_{k=0}^\infty {\bf H}^{*k}$. To complete the analysis we need to consider two cases separately.

{\bf Case a):} $C_i \geq 0$ for all $i$. 

For this case we have $\eta_- \equiv 0$, from where it follows that
$$\boldsymbol{\nu} = {\bf e} {\bf U} = (1,0) \sum_{k=0}^\infty \left( \begin{matrix} \eta_+ & 0 \\ 0 & \eta_+ \end{matrix} \right)^{*k} = (1, 0) \left( \begin{matrix} \sum_{i=1}^\infty \eta_+^{*k} & 0 \\ 0 & \sum_{k=0}^\infty \eta_+^{*k} \end{matrix} \right) = \left( \sum_{k=0}^\infty \eta_+^{*k}, 0 \right),$$
which in turn implies that
$$\breve{r}(t) = (\nu^{(+)}* \breve{g}_+)(t) = \sum_{k=0}^\infty (\breve{g}_+*\eta_+^{*k})(t).$$
Then, by the matrix version of the Key Renewal Theorem on the real line, Theorem 4 in \cite{Sgi_06}, 
$$\lim_{t \to \infty} e^{-\beta t} \int_{0}^{e^t} v^{\alpha+\beta-1} P(R > v) dv = \lim_{t \to \infty} \breve{r}(t) = \frac{1}{\mu} \int_{-\infty}^\infty \breve{g}_+(u) du \triangleq \frac{H_+}{\beta}.$$
Clearly, $H_+ \geq 0$ since the left-hand side of the preceding equation is positive, and thus, by 
Lemma \ref{L.Derivative},
$$P(R > t) \sim H_+ t^{-\alpha}, \qquad t \to \infty.$$
To derive the result for $P(R < -t)$, simply start by developing a telescoping sum for $P(R < -e^t)$ in \eqref{eq:telescoping}, define $r(t) = e^{\alpha t} P(R < - e^t)$ and follow exactly the same steps to obtain
$$\lim_{t \to \infty} e^{-\beta t} \int_{0}^{e^t} v^{\alpha+\beta-1} P(R <- v) dv = \frac{1}{\mu} \int_{-\infty}^\infty \breve{g}_-(u) du \triangleq \frac{H_-}{\beta}$$
and
$$P(R < -t) \sim H_- t^{-\alpha}, \qquad t \to \infty.$$
To compute the constants $H_+, H_-$ note that
\begin{align*}
H_\pm &=  \frac{\beta}{\mu} \int_{-\infty}^\infty \int_{-\infty}^u e^{-\beta(u-t)} g_\pm(t) \, dt \, du \\
&=  \frac{1}{\mu} \int_{-\infty}^\infty e^{\beta t} g_\pm(t)  \int_{t}^\infty \beta e^{-\beta u} \, du \, dt \\
&= \frac{1}{ \mu} \int_{-\infty}^\infty  g_\pm(t) \, dt \\
&= \frac{1}{ \mu} \int_{-\infty}^\infty e^{\alpha t} \left( P( (\pm 1) R > e^t) - E\left[ \sum_{j=1}^{N} 1( (\pm 1)C_{j} R > e^{t} ) \right]    \right) dt \\
&= \frac{1}{ \mu} \int_{0}^\infty v^{\alpha-1} \left( P((\pm 1) R > v) - E\left[ \sum_{j=1}^{N} 1((\pm 1) C_{j} R > v ) \right]    \right) dv. 
\end{align*}

{\bf Case b):} $P(C_j < 0) > 0$ for some $j \geq 1$. 

For this case we have that $\eta_-$ is nonzero. Also, note that the matrix
$${\bf H}((-\infty,\infty)) = \left( \begin{matrix} E\left[ \sum_{j=1}^N |C_j|^\alpha \Indicator(X_j = 1) \right] & E\left[ \sum_{j=1}^N |C_j|^\alpha \Indicator(X_j = -1) \right] \\
E\left[ \sum_{j=1}^N |C_j|^\alpha \Indicator(X_j = -1) \right] & E\left[ \sum_{j=1}^N |C_j|^\alpha \Indicator(X_j = 1) \right]   \end{matrix} \right) \triangleq \left( \begin{matrix} p & 	q \\ q & p \end{matrix} \right)$$
is irreducible and has eigenvalues $\{1, q - p\}$, and therefore spectral radius equal to one. Moreover, $(1, 1)$ and $(1, 1)^T$ are left and right eigenvalues, respectively, of ${\bf H}((-\infty, \infty))$ corresponding to eigenvalue one, and by assumption,
$$(1, 1) \int_{-\infty}^\infty x {\bf H}(dx) \left( \begin{matrix} 1 \\ 1 \end{matrix} \right) = 2 \left(  \int_{-\infty}^\infty x \eta_+(dx) + \int_{-\infty}^\infty x \eta_-(dx) \right) = 2 E\left[ \sum_{j=1}^N |C_j|^\alpha \log|C_j| \right] = 2\mu > 0.$$
Furthermore, since the matrix of measures {\bf H} is nonlattice, Theorem 4 in \cite{Sgi_06} gives
$$\lim_{t \to \infty} {\bf U} * {\bf g}(t) = \frac{(1, 1)^T (1, 1)}{2\mu} \int_{-\infty}^\infty {\bf g}(u) du = \frac{1}{2\mu} \left(  \begin{matrix} \int_{-\infty}^\infty (\breve{g}_+(u) + \breve{g}_-(u)) du \\ \int_{-\infty}^\infty (\breve{g}_+(u) + \breve{g}_-(u)) du  \end{matrix} \right),$$
from where it follows that
$$\lim_{t \to \infty} e^{-\beta t} \int_{0}^{e^t} v^{\alpha+\beta-1} P(R > v) dv = \lim_{t \to \infty} \breve{r}(t) =  \lim_{t \to \infty} {\bf e}  ({\bf U}* {\bf g})(t) = \frac{1}{2\mu} \int_{-\infty}^\infty (\breve{g}_+(u) + \breve{g}_-(u)) du \triangleq \frac{H}{\beta}.$$
Note that $H = (H_+ + H_-)/2$, and by Lemma \ref{L.Derivative},
$$P(R > t) \sim H t^{-\alpha}, \qquad t \to \infty.$$
To derive the result for $P(R < -t)$ simply start by defining $r(t) = e^{\alpha t}P(R < -e^t)$, which in this case leads to the same asymptotics as above, that is,
$$P(R < - t) \sim H t^{-\alpha}, \qquad t \to \infty.$$
Finally, we note, by using the representations for $H_+$ and $H_-$ from Case a), that
\begin{align*}
H &= \frac{1}{2\mu} \int_0^\infty v^{\alpha-1} \left( P( R > v) - E\left[ \sum_{j=1}^N 1(C_j R > v) \right] \right) dv \\
&\hspace{5mm} + \frac{1}{2\mu} \int_0^\infty v^{\alpha-1} \left( P( R < -v) - E\left[ \sum_{j=1}^N 1(C_j R < - v) \right] \right) dv \\
&= \frac{1}{2\mu} \int_0^\infty v^{\alpha-1} \left( P( |R| > v) - E\left[ \sum_{j=1}^N 1(|C_j R| >  v) \right] \right) dv.
\end{align*}
\end{proof}

\section{The linear recursion: $R = \sum_{i=1}^N C_i R_i + Q$} \label{S.LinearRec}

Motivated by the information ranking problem on the internet, e.g. Google's PageRank algorithm 
\cite{Jel_Olv_10, Jel_Olv_11, Volk_Litv_08}, in this section we apply the implicit renewal theory for trees developed in the previous section to the following linear recursion:
\begin{equation} \label{eq:Linear}
R \stackrel{\mathcal{D}}{=} \sum_{i=1}^N C_i R_i + Q,
\end{equation}
where $N  \in \mathbb{N} \cup \{\infty\}$,
$\{C_i\}_{i=1}^N$ are real valued random weights, $Q$ is a real valued random variable with  $P(Q \not = 0) > 0$
and $\{R_i\}_{i\geq 1}$ are iid copies of $R$, independent of $(N, C_1,\dots, C_N)$.
Note that the power tail of $R$ for the case $Q \geq 0$, $\{C_i \geq 0 \}$ was previously studied in \cite{Jel_Olv_11}, the critical homogeneous case $(Q \equiv 0)$ with $\{C_i\ge 0\}$ was considered in \cite{Liu_00} and \cite{Iksanov_04}.

The first result we need to establish is the existence and finiteness of a solution to \eqref{eq:Linear}. For the purpose of existence we will provide an explicit construction of a solution $R$ to \eqref{eq:Linear} on a tree.  Note that such constructed $R$ will be the main object of study of this section. 

Recall that throughout the paper the convention is to denote the random vector associated to the root node  $\emptyset$ by $(Q, N,C_1, \dots, C_N) \equiv (Q_\emptyset, N_\emptyset, C_{(\emptyset, 1)}, \dots, C_{(\emptyset, N_\emptyset)})$. 

We now define the process
\begin{equation} \label{eq:W_k}
W_0 =  Q, \quad W_n =  \sum_{{\bf i} \in A_n} Q_{{\bf i}} \Pi_{{\bf i}}, \qquad n \geq 1,
\end{equation}
on the weighted branching tree ${\mathcal T}_{Q, C}$, as constructed in Section \ref{S.ModelDescription}.

Define the process $\{R^{(n)}\}_{n \geq 0}$ according to
\begin{equation} \label{eq:R_nDef}
R^{(n)} = \sum_{k=0}^n W_k , \qquad n \geq 0,
\end{equation}
that is, $R^{(n)}$ is the sum of the weights of all the nodes on the tree up to the $n$th generation. It is not hard to see that $R^{(n)}$ satisfies the recursion
\begin{equation} \label{eq:LinearRecSamplePath} 
R^{(n)} = \sum_{j=1}^{N_\emptyset} C_{(\emptyset,j)} R^{(n-1)}_{j} + Q_{\emptyset} = \sum_{j=1}^{N} C_{j} R^{(n-1)}_{j} + Q, \qquad n \geq 1,
\end{equation}
where $\{R_{j}^{(n-1)}\}$ are independent copies of $R^{(n-1)}$ corresponding to the tree starting with individual $j$ in the first generation and ending on the $n$th generation; note that $R_j^{(0)} = Q_j$. Moreover, since the tree structure repeats itself after the first generation, $W_n$ satisfies
\begin{align}
W_n &= \sum_{{\bf i} \in A_n} Q_{{\bf i}} \Pi_{{\bf i}} \notag\\
&= \sum_{k = 1}^{N_\emptyset} C_{(\emptyset,k)}  
\sum_{(k,\dots,i_n) \in A_n} Q_{(k,\dots,i_n)} \prod_{j=2}^n C_{(k,\dots,i_j)}  \notag\\
&\stackrel{\mathcal{D}}{=} \sum_{k=1}^N C_k W_{(n-1),k},\label{eq:WnRec}
\end{align}
where $\{W_{(n-1),k}\}$ is a sequence of iid random variables independent of $(N, C_1, \dots, C_N)$ and having the same distribution as $W_{n-1}$. 

\begin{lem} \label{L.RnConvergence} 
If for some $0<\beta\le 1$, $E\left[|Q|^\beta\right]<\infty$, $E \left[ \sum_{j=1}^N |C_j|^\beta\right]<1$, then 
$R^{(n)}\rightarrow R$ a.s. as $n\rightarrow \infty$, where $E[|R|^\beta]<\infty$ and is given by
\begin{equation} \label{eq:ExplicitConstr}
R  \eqdef \sum_{n=0}^\infty W_n.
\end{equation}
\end{lem}

\begin{rem}
If $E[N] < 1$ the tree is finite a.s. and thus $R$ is finite a.s. for any choice of $Q$ and $\{C_i\}$. 
\end{rem}

\begin{proof}[Proof of Lemma \ref{L.RnConvergence}]
By Corollary~4 on p.~68 in \cite{ChowTeich1988} the a.s. convergence of $R^{(n)}$ will follow once we show that, in probability,
$$
\sup_{m>n}|R^{(m)}-R^{(n)}|\to 0, \quad \text{as  } n\rightarrow \infty.
$$ 
To this end, note that that for any $\epsilon>0$
\begin{align}
P\left( \sup_{m>n}|R^{(m)}-R^{(n)}| >\epsilon \right)
&\le P\left( \sup_{m>n}\sum_{i=n+1}^m |W_i| >\epsilon \right) \notag\\
&= P\left( \sum_{i=n+1}^\infty |W_i| >\epsilon \right)\notag\\
&\le \frac{1}{\epsilon^\beta}E\left[ \left( \sum_{i=n+1}^\infty |W_i| \right)^\beta \right] \notag\\
&\le \frac{1}{\epsilon^\beta}E\left[ \sum_{i=n+1}^\infty |W_i|^\beta\right],\label{eq:RnConvergence}
\end{align}
where the last inequality follows from the elementrary inequality 
$(\sum_i y_i)^\beta\le \sum_i y_i^\beta$ for $y_i\ge 0$ and $0<\beta\le 1$;
this elementary inequality is used repeatedly in the remainder of this proof and paper.
Now, the last sum can be easily evaluated since by Lemma \ref{L.MomentSmaller_1} below we have 
$$
E\left[ |W_i|^\beta\right]\le E\left[|Q|^\beta\right]\rho_\beta^i,
$$
where $\rho_\beta=E \left[ \sum_{j=1}^N |C_j|^\beta\right]$.
Therefore, by combining the preceding two inequalities we obtain
$$
P\left( \sup_{m>n}|R^{(m)}-R^{(n)}| >\epsilon \right)\le \frac{1}{\epsilon^\beta} \cdot \frac{E\left[|Q|^\beta\right]\rho_\beta^{n+1}}{1-\rho_\beta}\rightarrow 0
$$
as $n\rightarrow \infty$, which completes the proof of the a.s. convergence part. 
Thus, the infinite sum in \eqref{eq:ExplicitConstr}
 is properly defined and
 $$
 E[|R|^\beta]\le E\left[\sum_{i=0}^\infty |W_i|^\beta\right]=\frac{E\left[|Q|^\beta\right]}{1-\rho_\beta}<\infty.
 $$
\end{proof}

Furthermore, under the assumption of the preceding lemma, it is easy to see that the sum of all the absolute values of the 
weights on the tree are a.s. finite, i.e.,
$$
\sum_{n=0}^\infty \sum_{{\bf i}\in A_n} |Q_{\bf i}\Pi_{\bf i}|<\infty \quad \text{a.s.}
$$
Hence, it can be easily seen from the construction of $R$ on the tree, that it can be decomposed into the following identity
$$R = \sum_{j=1}^{N_\emptyset} C_{(\emptyset,j)} R_{j}^{(\infty)} + Q_\emptyset = \sum_{j=1}^N C_{j} R_j^{(\infty)} + Q ,$$
where $\{R_{j}\}$ are independent copies of $R$ corresponding to the infinite subtree starting with individual $j$ in the first generation.
The derivation provided above implies in particular the existence of a solution in distribution to \eqref{eq:Linear}. 
Moreover, we will show in the following section that, under additional technical conditions, $R$ is the unique solution.
The constructed $R$, as defined in \eqref{eq:ExplicitConstr}, is the main object of study in the remainder of this section. Note that, in view of the very recent work in \cite{Alsm_Mein_10b}, \eqref{eq:Linear} may have other stable law solutions that are not considered here.

\subsection{Moments of $W_n$ and $R$} \label{SS.MomentsLinear}

In order to establish the finiteness of moments of $W_n$ and $R$ let $A_{\mathcal{T}} = \bigcup_{n = 0}^\infty A_n$ and note that
\begin{align*}
&|W_n| \leq \sum_{{\bf i} \in A_n} |Q_{\bf i}| |\Pi_{\bf i} |, \qquad n \geq 1, \\
\text{and } &\quad |R| \leq \sum_{n=0}^\infty |W_n| \leq \sum_{{\bf i} \in A_{\mathcal{T}} } |Q_{\bf i}| |\Pi_{\bf i} |, 
\end{align*}
so Lemmas 4.2, 4.3 and 4.4 in \cite{Jel_Olv_11} apply and we immediately obtain the following results. Throughout the rest of the paper we use $\rho_\beta = E\left[ \sum_{i=1}^N |C_i|^\beta \right]$ and $\rho \equiv \rho_1$. 

\begin{lem} \label{L.MomentSmaller_1}
Let $0 < \beta \leq 1$. Then, for all $n \geq 0$,
$$E[ |W_n|^\beta ] \leq  E[ |Q|^\beta] \rho_\beta^{n}.$$
\end{lem}

\begin{lem} \label{L.GeneralMoment}
Let $\beta > 1$ and suppose $E\left[ \left( \sum_{i=1}^N |C_i| \right)^\beta \right] < \infty$,  $E[|Q|^\beta]< \infty$, and $\rho \vee \rho_\beta < 1$. Then, there exists a constant $K_\beta > 0$ such that for all $n \geq 0$, 
\begin{equation*} 
E[ |W_n|^\beta ] \leq K_\beta ( \rho \vee \rho_\beta  )^{n}.
\end{equation*}
\end{lem}

\begin{lem} \label{L.Moments_R}
Assume $E[|Q|^\beta] < \infty$ for some $\beta > 0$. In addition, suppose either (i) $\rho_\beta < 1$ if $0 < \beta < 1$, or (ii) $(\rho \vee \rho_\beta)  < 1$ and $E\left[\left( \sum_{i=1}^N |C_i| \right)^\beta\right] < \infty$ if $\beta \geq 1$. Then, $E[|R|^\gamma] < \infty$ for all $0  < \gamma \leq \beta$. Moreover, if $\beta \geq 1$, $R^{(n)} \stackrel{L_\beta}{\to} R$, where $L_\beta$ stands for convergence in $(E|\cdot|^\beta)^{1/\beta}$ norm. 
\end{lem}

\subsection{Asymptotic behavior} \label{SS.MainLinear}

We now characterize the tail behavior of the distribution of the solution $R$ to the nonhomogeneous equation \eqref{eq:Linear}, 
as defined by \eqref{eq:ExplicitConstr}.

\begin{thm} \label{T.LinearRecursion}
Let $(Q, N, C_1, \dots, C_N)$ be a random vector, with $N \in \mathbb{N} \cup \{\infty\}$, $\{C_i\}_{i=1}^N$ real valued weights, $Q$ a real valued random variable with $P(|Q| > 0) > 0$
and $R$ be the solution to \eqref{eq:Linear} given by  \eqref{eq:ExplicitConstr}. 
Suppose that there exists $j \geq 1$ with $P(N\ge j,|C_j|>0)>0$ such that the measure $P(\log |C_j|\in du, |C_j| > 0, N\ge j)$ is nonlattice, and 
that for some $\alpha > 0$,  $E[|Q|^\alpha] < \infty$, $E \left[ \sum_{i=1}^N |C_i|^\alpha \log |C_i| \right] > 0$ and  $ E \left[ \sum_{i=1}^N |C_i|^\alpha \right] = 1$. In addition, assume
\begin{enumerate} \renewcommand{\labelenumi}{\arabic{enumi})}
\item $E\left[ \sum_{i=1}^N |C_i| \right] < 1$ and $E\left[ \left( \sum_{i=1}^N |C_i| \right)^\alpha \right] < \infty$, if $\alpha > 1$; or,
\item $E\left[ \left( \sum_{i=1}^N |C_i|^{\alpha/(1+\epsilon)}\right)^{1+\epsilon} \right] < \infty$ for some $0 < \epsilon< 1$, if $0 < \alpha \leq 1$. 
\end{enumerate}
Then,
\begin{enumerate}
\item if $\{C_i \} \geq 0$ a.s.
$$P(R > t) \sim H_+ t^{-\alpha}, \qquad P(R < -t) \sim H_- t^{-\alpha}, \qquad t \to \infty,$$
where $H_\pm \geq 0$ are given by
\begin{align*}
H_\pm &= \frac{1}{E\left[ \sum_{i=1}^N |C_i|^\alpha \log |C_i|  \right] } \int_{0}^\infty v^{\alpha-1} \left( P((\pm 1) R > v) - E\left[ \sum_{i=1}^{N} 1((\pm 1) C_{i} R > v ) \right]    \right) dv \\
&= \frac{E\left[ \left(\left( \sum_{i=1}^N C_i R_i +Q \right)^\pm\right)^\alpha - \sum_{i=1}^N \left((C_i R_i )^\pm \right)^\alpha \right]}{\alpha E\left[ \sum_{i=1}^N 
|C_i|^\alpha \log |C_i|  \right] }.
\end{align*}
\item if $P(C_j < 0) > 0$ for some $j \geq 1$, 
$$P(R > t) \sim P(R < -t) \sim H t^{-\alpha}, \qquad t \to \infty,$$
where
\begin{align*}
H &= \frac{1}{2E\left[ \sum_{i=1}^N |C_i|^\alpha \log |C_i|  \right] } \int_{0}^\infty v^{\alpha-1} \left( P(|R| > v) - E\left[ \sum_{i=1}^{N} 1( |C_{i} R| > v ) \right]    \right) dv \\
&= \frac{E\left[ \left| \sum_{i=1}^N C_i R_i + Q \right|^\alpha - \sum_{i=1}^N |C_i R_i |^\alpha \right]}{2\alpha E\left[ \sum_{i=1}^N |C_i|^\alpha \log |C_i|  \right] }.
\end{align*}
\end{enumerate}
\end{thm}

\begin{rem}
 (i) When $\alpha > 1$, the condition $E\left[ \left( \sum_{i=1}^N |C_i| \right)^\alpha \right] < \infty$ is needed to ensure that the tails of $R$ are not dominated by $N$. In particular, if the $\{C_i\}$ are nonnegative iid and independent of $N$, the condition reduces to $E[N^\alpha] < \infty$ since $E[C^\alpha] < \infty$ is implied by the other conditions; see Theorems~4.2 and 5.4 in \cite{Jel_Olv_10}. Furthermore, when $0 < \alpha \leq 1$ the condition $E\left[ \left( \sum_{i=1}^N |C_i| \right)^\alpha \right] < \infty$ is redundant since $E\left[ \left( \sum_{i=1}^N |C_i| \right)^\alpha \right] \leq E\left[  \sum_{i=1}^N |C_i|^\alpha \right] = 1$, and the additional condition $E\left[ \left( \sum_{i=1}^N |C_i|^{\alpha/(1+\epsilon)} \right)^{1+\epsilon} \right] < \infty$ is needed. When the $\{C_i \}$ are nonnegative iid and independent of $N$ (given the other assumptions), the latter condition reduces to  $E[N^{1+\epsilon}] < \infty$, which is consistent with Theorem 4.2 in \cite{Jel_Olv_10}.  (ii) Note that the expressions for $H_\pm$ and $H$ given in terms of moments are more suitable for actually computing them, especially in the case of $\alpha$ being an integer (see Corollary 4.9 in \cite{Jel_Olv_11}).   When $\alpha$ is not an integer, we can derive bounds on $H_\pm$ and $H$ by using moment inequalities, e.g. in the case when $Q \geq 0$ and $\{C_i \geq 0\}$, the elementary inequality 
$\left( \sum_{i=1}^k x_i \right)^\alpha \ge \sum_{i=1}^k x_i^\alpha$ for $\alpha\ge1$ and $x_i \geq 0$, yields
$$
H_+ \ge \frac{E\left[ Q^\alpha  \right]}{\alpha E\left[ \sum_{i=1}^N C_i^\alpha \log C_i  \right] }>0.
$$
\end{rem}

Before giving the proof of Theorem \ref{T.LinearRecursion}, we state the following preliminary lemmas; their proofs are contained in Section \ref{SS.ProofsLinearRec}.  With some abuse of notation, we will use throughout the paper $\max_{1\leq i \leq N} x_i$ to denote $\sup_{1 \leq i < N+1} x_i$ in case $N = \infty$.

\begin{lem} \label{L.d_Moments_larger1}
Suppose $(N, C_1, \dots, C_N)$ is a random vector with $N \in \mathbb{N}$ and $\{C_i\}$ real valued random variables.  Let $\{R_i\}_{i \geq 1}$ be a sequence of iid real valued random variables having the same distribution as $R$, independent of $(N, C_1,\dots, C_N)$. Further assume $\sum_{i=1}^N \left| C_i R_i \right| < \infty$ a.s., $E\left[ \left( \sum_{i=1}^N |C_i| \right)^\beta \right] < \infty$ for some $\beta > 1$, and $E[|R|^\eta] < \infty$ for all $0 < \eta < \beta$. Then, for  $d(t)$ equal to any of the functions $t^+$, $t^-$ or $|t|$, 
$$E\left[ \left| d\left( \sum_{i=1}^N C_i R_i \right)^\beta -  \sum_{i=1}^N d(C_iR_i)^\beta \right| \right] < \infty.$$
\end{lem}

\begin{lem} \label{L.d_Moments_smaller1}
Suppose $(N, C_1, \dots, C_N)$ is a random vector with $N \in \mathbb{N}$ and $\{C_i\}$ real valued random variables.  Let $\{R_i\}_{i \geq 1}$ be a sequence of iid real valued random variables having the same distribution as $R$, independent of $(N, C_1,\dots, C_N)$. Further assume $\sum_{i=1}^N \left| C_i R_i \right| < \infty$ a.s., $E\left[ \sum_{i=1}^N |C_i|^\beta  \right] < \infty$, $E\left[ \left( \sum_{i=1}^N |C_i|^{\beta/(1+\epsilon)} \right)^{1+\epsilon} \right]$ for some $0 < \beta \leq 1$, $0 < \epsilon < 1$, and $E[|R|^\eta] < \infty$ for all $0 < \eta < \beta$. Then, for  $d(t)$ equal to any of the functions $t^+$, $t^-$ or $|t|$, 
$$E\left[ \left| d\left( \sum_{i=1}^N C_i R_i \right)^\beta -  \sum_{i=1}^N d(C_iR_i)^\beta \right| \right] < \infty.$$
\end{lem}

\begin{lem} \label{L.Max_Approx}
Suppose $(N, C_1, \dots, C_N)$ is a random vector, with $N \in \mathbb{N} \cup \{\infty\}$ and $\{C_i\}_{i=1}^N$ real valued weights, and let $\{R_i\}_{i \geq 1}$ be a sequence of iid random variables having the same distribution as $R$, independent of $(N,C_1,\dots, C_N)$. For $\alpha >0$, suppose that $\sum_{i=1}^N |C_i R_i|^\alpha < \infty$ a.s. and $E[|R|^\beta]< \infty$ for any $0 < \beta < \alpha$.  Furthermore, assume that  $E\left[ \left( \sum_{i=1}^N |C_i|^{\alpha/(1+\epsilon)}\right)^{1+\epsilon} \right] < \infty$ for some $0 < \epsilon< 1$. Then, 
\begin{align*}
0 &\leq \int_{0}^\infty \left( E\left[ \sum_{i=1}^N 1(T_i > t ) \right] - P\left( \max_{1\leq i \leq N} T_i > t \right) \right)  t^{\alpha -1} \, dt \\
&= \frac{1}{\alpha} E \left[  \sum_{i=1}^N  \left( T_i^+ \right)^\alpha - \left( \left( \max_{1\leq i \leq N} T_i \right)^+ \right)^\alpha   \right]  < \infty,
\end{align*}
where $T_i$ can be taken to be any of the random variables $C_i R_i$, $-C_i R_i$, or $|C_i R_i|$. 
\end{lem}

\begin{lem} \label{L.ExtraQ}
Let $(Q, N, C_1, \dots, C_N)$ be a random vector with $N \in \mathbb{N} \cup \{\infty\}$, $\{C_i\}_{i=1}^N$ real valued weights and $Q$ real valued, and let $\{R_i\}_{i \geq 1}$ be a sequence of iid random variables independent of $(Q, N,C_1,\dots, C_N)$. Suppose that for some $\alpha > 0$ we have $E[|Q|^\alpha] < \infty$, $E\left[ \left( \sum_{i=1}^N |C_i| \right)^\alpha \right] < \infty$,  $E[|R|^\beta] < \infty$ for any $0 < \beta < \alpha$, and $\sum_{i=1}^N |C_i R_i| < \infty$ a.s. Then, for $d(t)$ equal to any of the functions $t^+$, $t^-$ or $|t|$, 
$$E \left[ \left| d\left( \sum_{i=1}^N C_i R_i + Q  \right)^\alpha - d\left( \sum_{i=1}^N C_i R_i   \right)^\alpha \right|   \right]  < \infty .$$
\end{lem}

\bigskip

\begin{proof}[Proof of Theorem \ref{T.LinearRecursion}]
By Lemma \ref{L.Moments_R} we know that $E[|R|^\beta] < \infty$ for any $0 < \beta < \alpha$. To verify that $E\left[ \sum_{i=1}^N |C_i|^\gamma \right] < \infty$ for some $0 \leq \gamma < \alpha$ note that if $\alpha > 1$ we have, by the assumptions of the theorem and Jensen's inequality,
$$E\left[ \sum_{i=1}^N |C_i|^\gamma \right] \leq E\left[ \left( \sum_{i=1}^N |C_i| \right)^\gamma \right] \leq  \left( E\left[ \left( \sum_{i=1}^N |C_i| \right)^\alpha \right]  \right)^{\gamma/\alpha} < \infty$$
for any $1 \leq \gamma < \alpha$. If $0 < \alpha \leq 1$, then for $\gamma = \alpha(1+\epsilon/2)/(1+\epsilon) < \alpha$ we have
$$E\left[ \sum_{i=1}^N |C_i|^\gamma \right] \leq E\left[ \left( \sum_{i=1}^N |C_i|^{\alpha/(1+\epsilon)} \right)^{1+\epsilon/2} \right] \leq \left( E\left[ \left( \sum_{i=1}^N |C_i|^{\alpha/(1+\epsilon)} \right)^{1+\epsilon} \right] \right)^\frac{1+\epsilon/2}{1+\epsilon} < \infty.$$

The statement of the theorem with the first expressions for $H_+, H_-, H$ will follow from Theorem \ref{T.NewGoldie} once we prove that conditions \eqref{eq:Goldie_condition1} and \eqref{eq:Goldie_condition2} hold.  To this end define
$$R^* = \sum_{i=1}^N C_i R_i + Q,$$
and let $T_i$ be any of $C_i R_i$, $-C_iR_i$ or $|C_i R_i|$, depending on which condition is being verified; respectively, let $T^*$ be the corresponding $R^*$, $-R^*$ or $|R^*|$. Then, 
\begin{align*}
\left| P(T^*>t) - E\left[ \sum_{i=1}^N 1(T_i > t ) \right] \right|  &\leq \left| P(T^* > t) - P\left( \max_{1\leq i \leq N} T_i > t \right) \right|   \\
&\hspace{5mm} + \left| P\left( \max_{1\leq i \leq N} T_i > t  \right)   - E\left[ \sum_{i=1}^N 1(T_i > t ) \right]  \right|.
\end{align*}
To analyze the second absolute value, note that
\begin{align*}
&E\left[ \sum_{i=1}^N 1(T_i > t ) \right]  -  P\left(  \max_{1\leq i \leq N} T_i > t  \right)  \\
&= E\left[ \sum_{i=1}^N 1(T_i > t ) \right]  - E\left[ 1\left(  \max_{1\leq i \leq N} T_i > t  \right) \right] \geq 0.
\end{align*}
Now it follows that
\begin{align}
 \left| P(T^* > t) - E\left[ \sum_{i=1}^N 1(T_i > t ) \right] \right|  &\leq \left| P(T^* > t) - P\left( \max_{1\leq i \leq N} T_i > t \right) \right| \notag \\
&\hspace{5mm} + E\left[ \sum_{i=1}^N 1(T_i > t ) \right] - P\left( \max_{1\leq i \leq N} T_i > t \right)  . \label{eq:Term2}
\end{align}
Note that the integral corresponding to \eqref{eq:Term2} is finite by Lemma~\ref{L.Max_Approx} if we show that the assumptions of Lemma~\ref{L.Max_Approx} are satisfied when $\alpha > 1$. Note that in this case we can choose $\epsilon > 0$ such that $\alpha/(1+\epsilon) \geq 1$ and use the inequality 
\begin{equation} \label{eq:concaveSum}
\sum_{i=1}^k x_i^\beta \le \left( \sum_{i=1}^k x_i \right)^\beta  
\end{equation}
for $\beta \ge 1$, $x_i \geq 0$, $k \leq \infty$ to obtain
$$E\left[ \left( \sum_{i=1}^N |C_i|^{\alpha/(1+\epsilon)} \right)^{1+\epsilon} \right] \leq E\left[ \left( \sum_{i=1}^N |C_i| \right)^\alpha \right]  < \infty.$$
Therefore, it only remains to show that
\begin{align} 
&\int_0^\infty  \left| P(T^* > t) - P\left( \max_{1\leq i \leq N} T_i > t \right) \right| t^{\alpha-1} \, dt \notag \\
&\leq  \int_0^\infty  E\left[ \left| 1(T^* > t) - 1\left( \max_{1\leq i \leq N} T_i > t \right) \right| \right] t^{\alpha-1} \, dt < \infty. \label{eq:RecVsMax}
\end{align}

By Lemma \ref{L.NewIntegralIneq} in Section \ref{SS.ProofsLinearRec},
\begin{align}
\int_0^\infty  E\left[ \left| 1(T^* > t) - 1\left( \max_{1\leq i \leq N} T_i > t \right) \right| \right] t^{\alpha-1} \, dt &\leq \frac{1}{\alpha} E\left[ \left|   \left((T^*)^+\right)^\alpha - \left(\left( \max_{1\leq i\leq N} T_i \right)^+ \right)^\alpha \right|  \right] \notag \\
&\leq \frac{1}{\alpha} E\left[ \left|  ((T^*)^+)^\alpha -  \sum_{i=1}^N (T_i^+)^\alpha \right| \right] \label{eq:MainDifference} \\
&\hspace{5mm} + \frac{1}{\alpha} E\left[  \sum_{i=1}^N (T_i^+)^\alpha -  \left(\left( \max_{1\leq i\leq N} T_i \right)^+ \right)^\alpha  \right]. \label{eq:MaxDifference}
\end{align}
Note that \eqref{eq:MaxDifference} is finite by Lemma \ref{L.Max_Approx}, so it only remains to verifty that \eqref{eq:MainDifference} is finite. To see this let $d(t) = t^+$, $t^-$ or $|t|$ depending on whether $(T^*, T_i)$ is $(R^*, C_iR_i)$, $(-R^*, -C_iR_i)$ or $(|R^*|, |C_iR_i|)$, respectively, and let $S = \sum_{i=1}^N C_i R_i$. Then, the expectation in \eqref{eq:MainDifference} is equal to
\begin{align*}
E\left[ \left| d(S+Q)^\alpha - \sum_{i=1}^N d(C_iR_i)^\alpha \right| \right] &\leq E\left[ \left| d(S+Q)^\alpha - d(S)^\alpha  \right| \right]  + E\left[ \left| d(S)^\alpha - \sum_{i=1}^N d(C_iR_i)^\alpha \right| \right] .
\end{align*}
The first expectation on the right hand side is finite by Lemma~\ref{L.ExtraQ}, while the second one is finite by Lemmas~\ref{L.d_Moments_larger1} and \ref{L.d_Moments_smaller1}.

Finally, applying Theorem \ref{T.NewGoldie} gives the asymptotic expressions for $P(R > t)$ and $P(R < -t)$ with the integral representation of the constants $H_+$, $H_-$ and $H$. 

To obtain the expressions for $H_+$, $H_-$ and $H$ in terms of moments note that
\begin{align}
&\int_{0}^\infty v^{\alpha-1} \left( P( T^* > v) - E\left[ \sum_{j=1}^{N} 1( T_i > v ) \right]    \right) dv \notag \\
&= \int_0^\infty v^{\alpha-1}  E\left[  1(T^* > v) - \sum_{i=1}^N 1(T_i > v)  \right]  \, dv \notag \\ 
&= E \left[   \int_0^\infty v^{\alpha-1}  \left(  1(T^* > v)  -  \sum_{i=1}^N 1(T_i > v) \right) dv  \right] \label{eq:Fubini} \\
&= E \left[   \int_0^{(T^*)^+} v^{\alpha-1} dv  -  \sum_{i=1}^N \int_0^{T_i^+} v^{\alpha-1}  dv  \right] \label{eq:DiffIntegrals} \\
&= \frac{1}{\alpha} E\left[ \left( (T^*)^+ \right)^\alpha - \sum_{i=1}^N (T_i^+)^\alpha   \right] , \notag
\end{align}
where \eqref{eq:Fubini} is justified by Fubini's Theorem and the integrability of
\begin{align*}
v^{\alpha-1} \left| 1(T^* > v) -  \sum_{i=1}^{N} 1(T_i > v )  \right| &\leq v^{\alpha-1} \left| 1(T^* > v) -  1\left( \max_{1\leq i \leq N} T_i > v \right)  \right| \\
&\hspace{5mm} + v^{\alpha-1} \left( \sum_{i=1}^{N} 1(T_i > v ) - 1\left( \max_{1\leq i \leq N} T_i > v \right)  \right),
\end{align*}
which is a consequence of \eqref{eq:RecVsMax} and Lemma \ref{L.Max_Approx}; and \eqref{eq:DiffIntegrals} follows from the observation that 
$$v^{\alpha-1} 1(T^* > v) \qquad \text{and} \qquad v^{\alpha-1} \sum_{i=1}^N 1(T_i > v)$$
are each almost surely absolutely integrable with respect to $v$ as well. This completes the proof.
\end{proof}

\section{Proofs} \label{S.Proofs}

We separate the proofs corresponding to Sections \ref{S.Renewal} and \ref{S.LinearRec} into the following two subsections.

\subsection{Implicit renewal theorem on trees} \label{SS.ProofsRenewal}

This section contains the proofs of Lemma \ref{L.RenewalMeasure} and Theorem \ref{T.NewGoldie_Lattice}.

\begin{proof}[Proof of Lemma \ref{L.RenewalMeasure}]
To see that $\eta_+ + \eta_-$ is a probability measure note that 
\begin{align*}
\int_{-\infty}^{\infty} \eta_\pm(du) &=  \int_{-\infty}^\infty e^{\alpha u}E\left[ \sum_{j=1}^N 1(X_j = \pm 1, \log |C_j| \in du )  \right] \\
&= E\left[ \sum_{j=1}^N \int_{-\infty}^\infty e^{\alpha u} 1(X_j = \pm 1, \log |C_j| \in du )  \right] \qquad \text{(by Fubini's Theorem)} \\
&= E\left[ \sum_{j=1}^N 1(X_j = \pm 1) \int_{-\infty}^\infty e^{\alpha u} 1(\log |C_j| \in du)  \right]  \\
&= E\left[ \sum_{j=1}^N 1(X_j = \pm 1)  |C_j|^\alpha  \right]  
\end{align*}  
We then have that 
$$\int_{-\infty}^\infty \eta_+ (du) + \int_{-\infty}^\infty \eta_- (du) = E\left[ \sum_{j=1}^N |C_j|^\alpha \right] = 1.$$
Similarly, the mean of $\eta_+ + \eta_-$ is given by
$$\int_{-\infty}^\infty u \eta_+(du) + \int_{-\infty}^\infty u \eta_-(du) = E\left[ \sum_{j=1}^N |C_j|^\alpha \log |C_j|  \right] .$$

To show that \eqref{eq:convolution} holds we proceed by induction. 
For ${\bf i} \in A_n$, set $V_{{\bf i}} = \log |\Pi_{{\bf i}|n}|$, and 
let  $\mathcal{F}_{n}$, $n \geq 1$, denote the $\sigma$-algebra generated by $\{ (N_{\bf i}, C_{({\bf i},1)}, \dots, C_{({\bf i}, N_{\bf i})}): {\bf i} \in A_j, 0 \leq j \leq n-1\}$; $\mathcal{F}_0 = \sigma(\emptyset, \Omega)$, $\Pi_{{\bf i}|0} \equiv 1$. 
%
Let $Y_{\bf i} = \sgn(C_{\bf i})$. Hence, using this notation we derive
\begin{align*}
\mu_{n+1}^{(+)}((-\infty,t]) &= \int_{-\infty}^t e^{\alpha u} E\left[  \sum_{{\bf i} \in A_{n+1}} 1(X_{\bf i} = 1, V_{{\bf i}} \in du )   \right] \\
&= \int_{-\infty}^t e^{\alpha u}  E\left[  \sum_{{\bf i} \in A_{n}} \sum_{j=1}^{N_{\bf i}} \left\{ 1(X_{{\bf i}} = 1, Y_{({\bf i},j)} = 1, V_{{\bf i}} + \log |C_{({\bf i}, j)}| \in du )   \right. \right. \\
&\hspace{15mm} \left. \phantom{\sum_{j=1}^{N_i}} + \left. 1(X_{{\bf i}} = -1, Y_{({\bf i},j)} = -1, V_{{\bf i}} + \log |C_{({\bf i}, j)}| \in du )  \right\}  \right] \\
&= \int_{-\infty}^t e^{\alpha u}  E\left[  \sum_{{\bf i} \in A_{n}} \left\{ 1(X_{\bf i} = 1) \sum_{j=1}^{N_{\bf i}}  \right.    1(Y_{({\bf i},j)} = 1, V_{\bf i} + \log |C_{({\bf i}, j)}| \in du)  \right. \\
&\hspace{15mm} + \left. \left. 1( X_{\bf i} = -1 )  \sum_{j=1}^{N_{\bf i}}1(Y_{({\bf i},j)} = -1, V_{\bf i} + \log |C_{({\bf i}, j)}| \in du )   \right\}  \right] \\
&= \int_{-\infty}^t e^{\alpha u}  E\left[  \sum_{{\bf i} \in A_{n}} \left\{ 1(X_{\bf i} = 1)  E\left[ \left. \sum_{j=1}^{N_{\bf i}}  \right.    1(Y_{({\bf i},j)} = 1, V_{\bf i} + \log |C_{({\bf i}, j)}| \in du) \right| \mathcal{F}_n \right] \right. \\
&\hspace{15mm} + \left. \left.  \left.  1( X_{\bf i} = -1 )  E\left[ \sum_{j=1}^{N_{\bf i}}1(Y_{({\bf i},j)} = -1, V_{\bf i} + \log |C_{({\bf i}, j)}| \in du )   \right| \mathcal{F}_n \right]  \right\} \right]. 
\end{align*}
Using the independence of $(N_{{\bf i}}, C_{({\bf i},j)}, \dots, C_{({\bf i},j)})$ and $\mathcal{F}_n$ we obtain
$$E\left[ \left. \sum_{j=1}^{N_{\bf i}}      1(Y_{({\bf i},j)} = \pm 1, V_{\bf i} + \log |C_{({\bf i}, j)}| \in du) \right| \mathcal{F}_n \right]  = e^{-\alpha (u- V_{\bf i})} \eta_\pm (du - V_{\bf i}),$$
from where it follows that
\begin{align*}
\mu_{n+1}^{(+)}((-\infty, t]) &= \int_{-\infty}^t  E\left[ \sum_{{\bf i} \in A_n} \left\{ 1(X_{\bf i} = 1) e^{\alpha V_{\bf i} } \eta_+ (du - V_{\bf i}) + 1(X_{\bf i} = -1) e^{\alpha V_{\bf i}} \eta_- (du - V_{\bf i}) \right\} \right] \\
&=   E\left[ \sum_{{\bf i} \in A_n}  1(X_{\bf i} = 1) e^{\alpha V_{\bf i} }  \eta_+ ((-\infty, t - V_{\bf i}]) \right] + E\left[ \sum_{{\bf i} \in A_n} 1(X_{\bf i} = -1) e^{\alpha V_{\bf i}}  \eta_- ((-\infty, t - V_{\bf i}])  \right]  \\
&= \int_{-\infty}^\infty \eta_+ ((-\infty, t-v]) \mu_n^{(+)}(dv) + \int_{-\infty}^\infty \eta_- ((-\infty, t-v]) \mu_n^{(-)}(dv),
\end{align*}
and hence $\mu_{n+1}^{(+)}(dt) = (\eta_+*\mu_n^{(+)})(dt) +  (\eta_-*\mu_n^{(-)})(dt)$. The same arguments also give
$$\mu_{n+1}^{(-)}(dt) = (\eta_-*\mu_n^{(+)})(dt) +  (\eta_+*\mu_n^{(-)})(dt).$$
In matrix notation the last two equations can be written as
$$\left(  \mu_{n+1}^{(+)} , \, \mu_{n+1}^{(-)} \right) = (\mu_n^{(+)}, \, \mu_n^{(-)} ) * \left( \begin{matrix} 
\eta_+ & \eta_- \\ \eta_- & \eta_+ \end{matrix} \right) ,$$
and now the induction hypothesis gives the result.
\end{proof}

\bigskip

Before going into the proof of Theorem \ref{T.NewGoldie_Lattice} we need the following lattice analogue of the monotone density lemma.

\begin{lem} \label{L.DerivativeDiscrete}
Let $\alpha, \beta > 0$ and fix $t \in \mathbb{R}$. Suppose that $\int_{-\infty}^{t+\lambda n} e^{(\alpha + \beta) u } P(R > e^u) du \sim G(t) e^{\beta(t+\lambda n)}/\beta$ as $n \to \infty$, with $0 \leq G(t) < \infty$. If $H(t) = \lim_{h \to 0} (e^{\beta h}G(t+h)- G(t))/(\beta h)$ exists, then
$$P(R > e^{t+\lambda n} ) \sim H(t) e^{-\alpha(t+\lambda n)} , \qquad n \to \infty.$$
\end{lem}

\begin{proof}
Fix $0 < \delta, \epsilon < \min\{\eta, 1\}$. By assumption, for any $b > 1$, $\epsilon \in (0,1)$, and $n$ sufficiently large,
\begin{align*}
P(R > e^{t+\lambda n}) e^{(\alpha+\beta)(t+\lambda n)} \cdot \frac{(e^{(\alpha+\beta) \delta} - 1)}{\alpha+\beta}  &\geq \int_{t+\lambda n}^{t+\delta +\lambda n} e^{(\alpha+\beta)u} P(R > e^u) du  \\
&\geq \frac{(G(t+\delta)-\epsilon)}{\beta} e^{\beta(t+\delta +\lambda n)} - \frac{(G(t)+\epsilon)}{\beta} e^{\beta(t+\lambda n)} \\
&= \frac{e^{\beta(t+\lambda n)}}{\beta} \left( (G(t+\delta)-\epsilon) e^{\beta \delta} - G(t) - \epsilon  \right).
\end{align*}
Since $\epsilon$ was arbitrary, we can take the limit as $\epsilon \to 0$ to obtain
$$\liminf_{n \to \infty} P(R > e^{t+\lambda n}) e^{\alpha(t+\lambda n)} \geq \frac{\alpha+\beta}{e^{(\alpha+\beta)\delta} - 1} \cdot \frac{e^{\beta \delta} G(t+\delta) - G(t)}{\beta}.$$
Now take the limit as $\delta \downarrow 0$  to obtain
\begin{align*}
\lim_{\delta \downarrow 0} \frac{\alpha+\beta}{e^{(\alpha+\beta)\delta} - 1} \cdot \frac{e^{\beta \delta} G(t+\delta) - G(t)}{\beta} &= \lim_{\delta \downarrow 0} \frac{(\alpha+\beta)\delta}{e^{(\alpha+\beta)\delta} - 1} \cdot \lim_{\delta \downarrow 0} \frac{e^{\beta \delta} G(t+\delta) - G(t)}{\beta\delta} = H(t).  
\end{align*}
Similarly, one can prove that $\limsup_{t \to \infty} P(R > e^{t+\lambda n}) e^{\alpha(t+\lambda n)}  \leq H(t)$ by starting with the integral $\int_{t-\delta +\lambda n}^{t+\lambda n} e^{(\alpha+\beta)u} P\left(R > e^u\right) du$. 
%
\end{proof}

Now we proceed with the proof of Theorem \ref{T.NewGoldie_Lattice}. 

\begin{proof}[Proof of Theorem \ref{T.NewGoldie_Lattice}]
Define $\eta_+$, $\eta_-$ and ${\bf H}$ as in Lemma \ref{L.RenewalMeasure}. We first note that by assumption, 
\begin{align*}
\eta_+(dt) &= e^{\alpha t} E\left[ \sum_{i=1}^N 1( \sgn(C_i) = 1, \log |C_i| \in dt ) \right] \quad \text{and} \\
\eta_-(dt) &= e^{\alpha t} E\left[ \sum_{i=1}^N 1( \sgn(C_i) = -1, \log |C_i| \in dt ) \right]
\end{align*}
are both lattice measures on the lattice $L$. Then, according to Definition 5 in \cite{Sgi_06} (with $\alpha_1 = \alpha_2 = 0$), the matrix ${\bf H}$ is lattice with span $\lambda$. 

The proof of the theorem is identical to that of Theorem \ref{T.NewGoldie} up to the point where the matrix analogue of the Key Renewal Theorem on the real line, Theorem 4 in \cite{Sgi_06}, is used. 

{\bf Case a):} $C_i \geq 0$ for all $i$.

Applying Theorem 4 in \cite{Sgi_06} we obtain that for any $t \in \mathbb{R}$,
$$\lim_{n\to \infty} e^{-\beta (t+\lambda n)} \int_{-\infty}^{t+\lambda n} e^{(\alpha+\beta)u} P(R > e^u) du = \lim_{n \to \infty} \breve{r}(t+\lambda n) = \frac{\lambda}{\mu} \sum_{k=-\infty}^\infty \breve{g}_+(t+ k\lambda)\triangleq \frac{G_+(t)}{\beta}$$
and
$$\lim_{n \to \infty} e^{-\beta(t+\lambda n)} \int_{-\infty}^{t+\lambda n} e^{(\alpha+\beta)u} P(R < -e^u) dv = \frac{\lambda}{\mu} \sum_{k = -\infty}^\infty \breve{g}_- (t+k\lambda) \triangleq \frac{G_-(t)}{\beta}.$$
We now verify that the limit $\lim_{\delta \to 0} (e^{\beta \delta}G_\pm(t+\delta) - G_\pm(t))/\delta$ exists. To do this first define the function $H_\pm(t) \triangleq \frac{\lambda}{\mu} \sum_{k=-\infty}^{\infty} g_\pm(t+k\lambda)$ and fix $0 < \delta < \lambda$. Then,
\begin{align*}
\frac{e^{\beta \delta} G_\pm (t+\delta) - G_\pm(t)}{\beta\delta} &= \frac{\lambda}{\delta\mu} \sum_{k=-\infty}^\infty \left(e^{\beta \delta} \breve{g}_\pm(t+\delta + k\lambda) - \breve{g}_\pm(t+k\lambda) \right)  \\
&= \frac{\lambda}{\delta\mu} \sum_{k=-\infty}^\infty  \int_{t+k\lambda}^{t+\delta +k\lambda} e^{-\beta(t+k\lambda-u)} g_\pm(u) du     \\
&=\frac{\lambda}{\delta\mu} \sum_{k=-\infty}^\infty  \int_{0}^{\delta} e^{\beta v} g_\pm(v+t+k\lambda) dv   \\
&= \frac{1}{\delta} \int_0^\delta e^{\beta v} H_\pm(v+t) dv \\
&= \frac{e^{-\beta t}}{\delta} \int_t^{t+\delta} e^{\beta u} H_\pm(u) du , 
\end{align*}
where the rearrangement of summands in the first equality is justified by the absolute summability of the expressions, and the exchange of the integral and sum in the fourth equality is justified by Fubini's theorem and the observation that by \eqref{eq:Goldie_condition1_Discrete} and \eqref{eq:Goldie_condition2_Discrete}
$$\sum_{k=-\infty}^\infty \int_0^\delta e^{\beta v} |g_\pm(v+t+k\lambda)| dv \leq e^{\beta\lambda} \sum_{k=-\infty}^\infty \int_0^\lambda   |g_\pm(v+t+k\lambda)| dv = e^{\beta \lambda} \int_{-\infty}^{\infty}    |g_\pm(u)| du< \infty.$$ 
Similarly, 
$$\frac{e^{-\beta \delta} G_\pm (t-\delta) - G_\pm(t)}{-\beta\delta} =  \frac{e^{-\beta t}}{\delta} \int_{t-\delta}^t e^{\beta u} H_\pm(u) du.$$ 
Taking the limit as $\delta \to 0$ and using the Lebesgue differentiation theorem gives
$$\lim_{h \to 0} \frac{e^{\beta h} G_\pm (t+h) - G_\pm(t)}{\beta h} = H_\pm(t)$$ 
for almost every $t \in \mathbb{R}$. 

Next, by using Lemma \ref{L.DerivativeDiscrete} we obtain 
$$P(R > e^{t+\lambda n}) \sim H_+(t) e^{-\alpha(t+\lambda n)}, \qquad n \to \infty,$$
and
$$P(R < -e^{t+\lambda n}) \sim H_-(t) e^{-\alpha(t+\lambda n)}, \qquad n \to \infty.$$


{\bf Case b):} $P(C_j < 0) > 0$ for some $j \geq 1$. 

Applying Theorem 4 in \cite{Sgi_06} we obtain that for any $t \in \mathbb{R}$,
$$\lim_{n\to \infty} e^{-\beta (t+\lambda n)} \int_0^{e^{t+\lambda n}} v^{\alpha+\beta-1} P(R > v) dv = \lim_{n \to \infty} \breve{r}(t+\lambda n) = \frac{\lambda}{2\mu} \sum_{k=-\infty}^\infty (\breve{g}_+(t + k\lambda) + \breve{g}_-(t+k\lambda) ) \triangleq \frac{G(t)}{\beta}.$$
and
$$\lim_{n\to \infty} e^{-\beta (t+\lambda n)} \int_0^{e^{t+\lambda n}} v^{\alpha+\beta-1} P(R > v) dv = \frac{\lambda}{2\mu} \sum_{k=-\infty}^\infty (\breve{g}_+(t + k\lambda) + \breve{g}_-(t+k\lambda) ) \triangleq \frac{G(t)}{\beta},$$
where $G(t) = (G_+(t) + G_-(t))/2$. By using Lemma \ref{L.DerivativeDiscrete} we obtain (for almost every $t \in \mathbb{R}$)
$$P(R > e^{t+\lambda n}) \sim H(t) e^{-\alpha(t+\lambda n)}, \qquad n \to \infty,$$
where $H(t) = (H_+(t)+H_-(t))/2$.  
\end{proof}

\subsection{The linear recursion: $R = \sum_{i=1}^N C_i R_i + Q$} \label{SS.ProofsLinearRec}

In this section we give the proofs of Lemmas \ref{L.d_Moments_larger1}$-$\ref{L.ExtraQ}. We also state and prove an analogue of Lemma 4.1 in \cite{Jel_Olv_11} for the positive parts of general random variables, which will be used in the proofs of the lemmas mentioned above, and a version of Lemma 9.4 in \cite{Goldie_91} needed in the proof of Theorem \ref{T.LinearRecursion}.

\begin{lem} \label{L.Alpha_Moments}
For any $k \in \mathbb{N} \cup \{\infty\}$ let $\{D_i \}^k_{i=1}$ be a sequence of real valued random variables and let $\{Y_i\}_{i = 1}^k$ be a sequence of real valued iid random variables having the same distribution as $Y$, independent of the $\{D_i\}$. For $\beta > 1$ set $p = \lceil \beta \rceil \in \{2, 3, 4, \dots\}$, and if $k = \infty$ assume that $ \sum_{i=1}^\infty |D_i Y_i| < \infty$ a.s. Then,
$$E\left[ \left( \sum_{i=1}^k (D_i Y_i)^+ \right)^\beta - \sum_{i=1}^k ((D_iY_i)^+)^\beta \right] \leq E\left[ |Y|^{p-1} \right]^{\beta/(p-1)} E\left[ \left(\sum_{i=1}^k |D_i| \right)^\beta  \right].$$
\end{lem}

{\sc Remark:} Note that the preceding lemma does not exclude the case when $E \left[ \left( \sum_{i=1}^k \left(D_i Y_i \right)^+ \right)^\beta  \right] = \infty$ but \linebreak $E\left[ \left( \sum_{i=1}^k  \left(D_i Y_i  \right)^+\right)^\beta - \sum_{i=1}^k ((D_iY_i)^+)^\beta \right] < \infty$.

\begin{proof}[Proof of Lemma \ref{L.Alpha_Moments}]
Let $p = \lceil \beta \rceil \in \{2,3,\dots\}$ and $\gamma = \beta/p \in (\beta/(\beta+1), 1]$. Suppose first that $k \in \mathbb{N}$ and define  $A_p(k) = \{ (j_1, \dots, j_k) \in \mathbb{N}^k: j_1 + \dots + j_k = p, 0 \leq j_i < p\}$. Then, for any sequence of nonnegative numbers $\{ y_i \}_{i \geq 1}$ we have
\begin{align}
\left( \sum_{i=1}^k y_i \right)^\beta &= \left( \sum_{i=1}^k y_i \right)^{p \gamma} \notag \\
&= \left( \sum_{i=1}^k y_i^p + \sum_{(j_1,\dots,j_k) \in A_p(k)} \binom{p}{j_1,\dots,j_k} y_1^{j_1} \cdots y_k^{j_k} \right)^\gamma \notag  \\
&\leq \sum_{i=1}^k y_i^{p\gamma} + \left( \sum_{(j_1,\dots,j_k) \in A_p(k)} \binom{p}{j_1,\dots,j_k} y_1^{j_1} \cdots y_k^{j_k} \right)^\gamma, \label{eq:scalarBound}
\end{align}
where for the last step we used the well known inequality $\left( \sum_{i=1}^k x_i \right)^\gamma \leq \sum_{i=1}^k x_i^\gamma$ for $0 < \gamma \leq 1$ and $x_i \geq 0$. We now use the conditional Jensen's inequality to obtain 
\begin{align*}
&E\left[ \left( \sum_{i=1}^k (D_i Y_i)^+ \right)^\beta - \sum_{i=1}^k ((D_iY_i)^+)^{\beta} \right] \\
&\leq E\left[  \left( \sum_{(j_1,\dots,j_k) \in A_p(k)} \binom{p}{j_1,\dots,j_k} ((D_1Y_1)^+)^{j_1} \cdots ((D_k Y_k)^+)^{j_k} \right)^\gamma \right] \qquad \text{(by \eqref{eq:scalarBound})} \\
&\leq    E\left[ \left(  E\left[ \left. \sum_{(j_1,\dots,j_k) \in A_p(k)} \binom{p}{j_1,\dots,j_k} |D_1Y_1|^{j_1} \cdots |D_k Y_k|^{j_k} \right| D_1,\dots, D_k \right] \right)^\gamma \right]  \\
&= E \left[ \left(  \sum_{(j_1,\dots,j_k) \in A_p(k)} \binom{p}{j_1,\dots,j_k} |D_1|^{j_1} \cdots |D_k|^{j_k} E\left[ \left. |Y_1|^{j_1} \cdots |Y_k|^{j_k} \right| D_1,\dots, D_k \right] \right)^\gamma  \right].
\end{align*}
The rest of the proof is essentially the same as that of Lemma 4.1 in \cite{Jel_Olv_11}, and is therefore omitted.
\end{proof}

\bigskip

\begin{proof}[Proof of Lemma \ref{L.d_Moments_larger1}]
Suppose first that $d(t) = t^+$ and let $S_+ = \sum_{i=1}^N (C_i R_i)^+$, $S_- = \sum_{i=1}^N (C_i R_i)^-$, and $S = S_+ - S_-$, then
\begin{align}
&E\left[ \left| \left( \left( \sum_{i=1}^N C_i R_i\right)^+ \right)^\beta - \sum_{i=1}^N ((C_iR_i)^+)^\beta \right| \right] \notag \\
&\leq E\left[  \sum_{i=1}^N ((C_iR_i)^+)^\beta \Indicator(S_+ \leq S_-) \right] +  E\left[ \left| ( S_+ - S_-)^\beta - S_+^\beta \right| \Indicator(S_+ > S_-) \right] \label{eq:PosSum} \\
&\hspace{5mm} + E\left[ \left| S_+^\beta - \sum_{i=1}^N ((C_iR_i)^+)^\beta \right|\right] .\label{eq:PosPartsLemma}
\end{align}

Note that \eqref{eq:PosPartsLemma} is finite by Lemma \ref{L.Alpha_Moments}.  The first expectation in \eqref{eq:PosSum} can be bounded as follows
\begin{align}
E\left[  \sum_{i=1}^N ((C_iR_i)^+)^\beta \Indicator(S_+ \leq S_-) \right] &= E\left[ \sum_{i=1}^N E\left[ \left.  ((C_iR_i)^+)^\beta \Indicator(S_+ \leq S_-) \right| N, C_1, \dots, C_N \right] \right] \notag \\
&= E\left[ \sum_{i=1}^N E\left[ \left.  (C_iR_i)^\beta \Indicator\left(0 < C_i R_i \leq -S + C_i R_i \right) \right| N, C_1, \dots, C_N \right] \right]. \label{eq:BetaIndicator} 
\end{align}
When $1 < \beta \leq 2$, we have that \eqref{eq:BetaIndicator} is bounded by
\begin{align}
&E\left[ \sum_{i=1}^N E\left[ \left.  |C_iR_i| |S-C_iR_i|^{\beta-1} \right| N, C_1, \dots, C_N \right] \right] \notag \\
&= E\left[ |R| \right] E\left[ \sum_{i=1}^N |C_i|  E\left[ \left. |S-C_iR_i|^{\beta-1} \right| N, C_1, \dots, C_N \right] \right] \label{eq:IndepProd1} \\
&\leq E\left[ |R| \right] E\left[ \sum_{i=1}^N |C_i|  \left( E\left[ \left. |S-C_iR_i| \right| N, C_1, \dots, C_N \right] \right)^{\beta-1} \right] \label{eq:JensenConcave} \\
&\leq E\left[ |R| \right]^\beta E\left[ \sum_{i=1}^N |C_i|  \left(   \sum_{j=1}^N |C_j|  \right)^{\beta-1} \right] \notag \\
&= E\left[ |R| \right]^\beta E\left[  \left(   \sum_{j=1}^N |C_j|  \right)^{\beta} \right] < \infty, \notag
\end{align}
where in \eqref{eq:IndepProd1} we used the conditional independence of $C_iR_i$ and $S - C_iR_i$ and in \eqref{eq:JensenConcave} we used Jensen's inequality. Now, when $\beta > 2$ \eqref{eq:BetaIndicator} is bounded by
\begin{align}
&E\left[ \sum_{i=1}^N E\left[ \left.  |C_iR_i|^{\beta-1} |S-C_iR_i| \right| N, C_1, \dots, C_N \right] \right] \notag \\
&= E\left[ |R|^{\beta-1}  \right] E\left[ \sum_{i=1}^N |C_i|^{\beta-1}  E\left[ \left. |S-C_iR_i| \right| N, C_1, \dots, C_N \right] \right] \label{eq:IndepProd2} \\
&\leq E\left[ |R|^{\beta-1}  \right] E[|R|] E\left[ \sum_{i=1}^N |C_i|^{\beta-1}    \sum_{j=1}^N |C_j| \right] \notag \\
&\leq E\left[ |R|^{\beta-1}  \right] E[|R|] E\left[ \left( \sum_{i=1}^N |C_i| \right)^{\beta-1}    \sum_{j=1}^N |C_j| \right] < \infty, \notag
\end{align}
where in \eqref{eq:IndepProd2} we used the conditional independence of $C_iR_i$ and $S - C_i R_i$. 

For the second expectation in \eqref{eq:PosSum} we use the elementary inequality
$$|x^\beta - y^\beta| \leq \beta (x \vee y)^{\beta-1} |x-y|$$
for any $x,y \geq 0$ to obtain that
\begin{align}
&E\left[ \left| ( S_+ - S_-)^\beta - S_+^\beta \right| \Indicator(S_+ > S_-) \right] \\
&\leq \beta E\left[ S_+^{\beta-1} S_- \right] \notag \\
&= \beta E\left[ \sum_{i=1}^N E\left[ \left. S_+^{\beta-1} (C_iR_i)^-  \right| N, C_1, \dots, C_N \right] \right] \notag \\
&= \beta E\left[ \sum_{i=1}^N E\left[ \left. \left( S_+ - (C_i R_i)^+ \right)^{\beta-1} (C_i R_i)^-  \right| N, C_1, \dots, C_N \right] \right] \notag \\
&= \beta E\left[ \sum_{i=1}^N E\left[ \left. \left( S_+ - (C_i R_i)^+ \right)^{\beta-1} \right| N, C_1,\dots, C_N \right] E\left[ \left. (C_i R_i)^-  \right| N, C_1, \dots, C_N \right] \right] \notag \\
&\leq \beta E[|R|] E\left[ \sum_{i=1}^N |C_i| E\left[ \left.  S_+^{\beta-1} \right| N, C_1,\dots, C_N \right]  \right], \label{eq:MixedExp}
\end{align} 
where in the last equality we used the conditional independence of $(S_+-(C_iR_i)^+)^{\beta-1}$ and $(C_iR_i)^-$. To see that \eqref{eq:MixedExp} is finite note that if $1 < \beta \leq 2$, Jensen's inequality gives 
\begin{align*}
E\left[ \sum_{i=1}^N |C_i| E\left[ \left.  S_+^{\beta-1} \right| N, C_1,\dots, C_N \right]  \right] &\leq E\left[ \sum_{i=1}^N |C_i| \left( E\left[ \left.  S_+ \right| N, C_1,\dots, C_N \right] \right)^{\beta-1} \right] \\
&\leq E[|R|]^{\beta-1} E\left[ \sum_{i=1}^N |C_i|  \left( \sum_{j=1}^N |C_j| \right)^{\beta-1}   \right] < \infty.
\end{align*}
And if $\beta > 2$, we use Lemma \ref{L.Alpha_Moments} to obtain, for $p = \lceil \beta-1 \rceil$, 
\begin{align*}
E\left[ \left.  S_+^{\beta-1} \right| N, C_1,\dots, C_N \right] &\leq E\left[ \left.  \sum_{j=1}^N ((C_jR_j)^+)^{\beta-1} \right| N, C_1,\dots, C_N \right]  + E\left[|R|^{p-1}\right]^{(\beta-1)/(p-1)} \left( \sum_{j=1}^N |C_j| \right)^{\beta-1} \\
&\leq E\left[|R|^{\beta-1}\right] \sum_{j=1}^N |C_j|^{\beta-1} + E\left[|R|^{p-1}\right]^{(\beta-1)/(p-1)} \left( \sum_{j=1}^N |C_j| \right)^{\beta-1} \\
&\leq \left( ||R||_{\beta-1}^{\beta-1} + || R ||_{p-1}^{\beta-1} \right) \left( \sum_{j=1}^N |C_j| \right)^{\beta-1},
\end{align*}
where $|| \cdot ||_r = \left( E\left[ | \cdot |^r \right] \right)^{1/r}$. Next, using the monotonicity of $|| \cdot ||_r$ it follows that

$$E\left[ \sum_{i=1}^N |C_i| E\left[ \left.  S_+^{\beta-1} \right| N, C_1,\dots, C_N \right]  \right]  \leq 2 E\left[ |R|^{\beta-1} \right]  E\left[ \sum_{i=1}^N |C_i| \left( \sum_{j=1}^N |C_j| \right)^{\beta-1} \right] < \infty.$$
This completes the proof for $d(t) = t^+$.  To obtain the same result for $d(t) = t^-$ simply note that
$$E\left[ \left| \left( \left( \sum_{i=1}^N C_i R_i \right)^- \right)^\beta - \sum_{i=1}^N ((C_i R_i)^-)^\beta \right| \right] = E\left[ \left| \left( \left( \sum_{i=1}^N (- C_i R_i) \right)^+ \right)^\beta - \sum_{i=1}^N ((-C_i R_i)^+)^\beta \right| \right] $$
and apply the result for $d(t) = t^+$.  

Finally, for $d(t) = |t|$, we use the fact that $|x|^\beta = (x^+)^\beta + (x^-)^\beta$ for any $x \in \mathbb{R}$ to obtain
\begin{align*}
E\left[ \left| \left| \sum_{i=1}^N C_i R_i \right|^\beta - \sum_{i=1}^N |C_i R_i|^\beta \right| \right] 
&= E\left[ \left| (S^+)^\beta + (S^-)^\beta - \sum_{i=1}^N \left( ((C_i R_i)^+)^\beta + ((C_iR_i)^-)^\beta \right) \right| \right] 
\end{align*}
which is finite by the previous cases $d(t)= t^+$ and $d(t) = t^-$. 
\end{proof}

\bigskip

\begin{proof}[Proof of Lemma \ref{L.d_Moments_smaller1}]
From the proof of Lemma \ref{L.d_Moments_larger1} we see that it is enough to prove the result for $d(t) = t^+$. Let $S_+ = \sum_{i=1}^N (C_iR_i)^+$, $S_- = \sum_{i=1}^N (C_i R_i)^-$ and $S = S_+ - S_-$. Since $0 < \beta \leq 1$, we have
$$\left( \left( \sum_{i=1}^k y_i \right)^+ \right)^\beta \leq \left( \sum_{i=1}^k (y_i)^+ \right)^\beta \leq \sum_{i=1}^k ((y_i)^+)^\beta$$
for any real numbers $\{y_i\}$ and any $k \in \mathbb{N} \cup \{ \infty \}$. Hence,
\begin{align}
0 &\leq E\left[ \sum_{i=1}^N ((C_iR_i)^+)^\beta - \left(\left( \sum_{i=1}^N C_i R_i \right)^+ \right)^\beta \right] \notag \\
&= E\left[ \sum_{i=1}^N ((C_iR_i)^+)^\beta \Indicator(S_+ \leq S_-) \right] +  E\left[ \left( \sum_{i=1}^N ((C_iR_i)^+)^\beta - S_+^\beta \right) \Indicator(S_+ > S_-) \right] \label{eq:ConcaveSum1} \\
&\hspace{5mm} + E\left[ \left( S_+^\beta - (S_+-S_-)^\beta \right) \Indicator(S_+ > S_-) \right]. \label{eq:ConcaveSum2}
\end{align}

The first expectation in \eqref{eq:ConcaveSum1} can be bounded as follows. Let $a = \beta/(1+\epsilon)$ and $b = \epsilon\beta/(1+\epsilon)$
\begin{align*}
E\left[ \sum_{i=1}^N ((C_iR_i)^+)^\beta \Indicator(S_+ \leq S_-) \right]  &= E\left[ \sum_{i=1}^N E\left[ \left. ((C_iR_i)^+)^\beta \Indicator(0 < C_iR_i \leq -S + C_iR_i) \right| N, C_1,\dots, C_N \right] \right] \\
&\leq E\left[ \sum_{i=1}^N E\left[ \left. |C_iR_i|^{a} |S-C_iR_i|^{b} \right| N, C_1,\dots, C_N \right] \right] \\
&= E\left[ |R|^{a} \right] E\left[ \sum_{i=1}^N |C_i|^{a}  E\left[ \left. |S-C_iR_i|^{a \cdot \frac{b}{a}} \right| N, C_1,\dots, C_N \right] \right]  \\
&\leq E\left[ |R|^{a} \right]  E\left[ \sum_{i=1}^N |C_i|^{a} \left( E\left[ \left. \sum_{j=1}^N |C_jR_j|^{a} \right| N, C_1,\dots, C_N \right]  \right)^{\frac{b}{a}} \right] \\
&= \left( E\left[ |R|^{a} \right] \right)^{1+b/a} E\left[ \sum_{i=1}^N |C_i|^a \left( \sum_{j=1}^N |C_j|^{a} \right)^\frac{b}{a} \right] \\
&=  \left( E\left[ |R|^{\beta/(1+\epsilon)} \right] \right)^{1+\epsilon} E\left[  \left( \sum_{i=1}^N |C_i|^{\beta/(1+\epsilon)} \right)^{1+\epsilon} \right]  < \infty, 
\end{align*}
where in the second equality we used the conditional independence of $C_i R_i$ and $S - C_iR_i$. 

To analyze the expectation in \eqref{eq:ConcaveSum2} note that since $|x^\beta - y^\beta| \leq |x - y|^\beta$ for any $x,y \geq 0$, it follows that
\begin{align*}
E\left[ \left( S_+^\beta - (S_+-S_-)^\beta \right) \Indicator(S_+ > S_-) \right] &\leq E\left[ S_-^\beta  \Indicator(S_+ > S_-) \right] \leq E\left[ \sum_{i=1}^N ((C_iR_i)^-)^\beta \Indicator(S_- \leq S_+) \right],
\end{align*}
which is finite by the same arguments used above.

Finally, to analyze the second expectation in \eqref{eq:ConcaveSum1}, note that it is bounded by
\begin{align*}
E\left[  \sum_{i=1}^N ((C_iR_i)^+)^\beta - S_+^\beta  \right] &\leq E\left[  \sum_{i=1}^N ((C_iR_i)^+)^\beta - 
\left( \max_{1\leq i \leq N} (C_iR_i)^+\right)^\beta   \right] + E\left[  \left| \left( \max_{1\leq i \leq N} (C_iR_i)^+\right)^\beta - S_+^\beta \right|  \right] \\
&\leq 2 E\left[  \sum_{i=1}^N ((C_iR_i)^+)^\beta - 
\left( \max_{1\leq i \leq N} (C_iR_i)^+\right)^\beta   \right]  ,
\end{align*}
which is finite by Lemma \ref{L.Max_Approx}.
\end{proof}

\bigskip

\begin{proof}[Proof of Lemma \ref{L.Max_Approx}]
Let $T_i$ be any of the random variables $C_i R_i$, $-C_i R_i$, or $|C_i R_i|$ and note that the integral is positive since 
\begin{align*}
P\left( \max_{1\leq i \leq N} T_i > t \right) = E\left[   1\left(  \max_{1\leq i \leq N} T_i > t \ \right)  \right] &\leq E\left[   \sum_{i=1}^N 1\left( T_i > t   \right)  \right] .
\end{align*}
To see that the integral is equal to the expectation involving the $\alpha$-moments note that
\begin{align*}
&\int_{0}^\infty \left( E\left[ \sum_{i=1}^N 1(T_i > t ) \right] - P\left( \max_{1\leq i \leq N} T_i > t \right) \right)  t^{\alpha -1} \, dt \\
&= \int_0^\infty \left( E\left[ \sum_{i=1}^N 1(T_i > t)   -  1\left(\max_{1\leq i \leq N} T_i > t \right)  \right] \right)  t^{\alpha -1} \, dt \\
&= E\left[  \int_0^\infty \left(   \sum_{i=1}^N  1(T_i > t)  - 1\left(\max_{1\leq i \leq N} T_i > t \right)  \right)  t^{\alpha -1} \, dt  \right] \qquad \text{(by Fubini's Theorem)} \\
&= E\left[   \sum_{i=1}^N  \frac{1}{\alpha} (T_i^+)^{\alpha}  - \frac{1}{\alpha} \left( \left( \max_{1\leq i \leq N} T_i \right)^+\right)^{\alpha}  \right] ,
\end{align*}
where the last equality is justified by the assumption that $\sum_{i=1}^N |T_i|^\alpha < \infty$ a.s.

The rest of the proof is essentially the same as that of Lemma 4.7 in \cite{Jel_Olv_11} and is therefore omitted.
\end{proof}

\begin{proof}[Proof of Lemma \ref{L.ExtraQ}]
Let $S = \sum_{i=1}^N  C_i R_i$ and suppose first that $d(t) = t^+$. If $0 < \alpha \leq 1$, then we can use the inequality $|x^\alpha - y^\alpha| \leq |x-y|^\alpha$ for all $x,y \geq 0$ to obtain
\begin{align*}
E\left[ \left| ((S+Q)^+)^\alpha - (S^+)^\alpha \right| \right] &\leq  E\left[ \left| (S+Q)^+ - S^+ \right|^\alpha \right] \\
&= E\left[ \left( (S+Q)^+- S^+ \right)^\alpha \Indicator(Q \geq 0) \right] + E\left[ \left( S - (S+Q) \right)^\alpha \Indicator(Q < 0 \leq S+Q) \right] \\
&\hspace{5mm} + E\left[ (S^+)^\alpha \Indicator(Q < 0, S+Q < 0) \right] \\
&\leq E\left[ (Q^+)^\alpha \Indicator(Q \geq 0) \right] + E\left[ \left( -Q \right)^\alpha \Indicator(Q < 0 \leq S+Q) \right] \\
&\hspace{5mm} + E\left[ ((-Q)^+)^\alpha \Indicator(Q< 0, S+Q < 0) \right] \\
&\leq E[|Q|^\alpha] < \infty.
\end{align*}
If $\alpha > 1$ we use the inequality  
$$(x+t)^\kappa \leq \begin{cases}
x^\kappa + t^\kappa, & 0 < \kappa \leq 1, \\
x^\kappa + \kappa (x+t)^{\kappa-1} t, & \kappa > 1,
\end{cases}$$
for any $x,t \geq 0$. Let $p = \lceil \alpha \rceil$, apply the second inequality $p-1$ times and then the first one to obtain 
$$(x+t)^\alpha \leq x^\alpha + \alpha (x+t)^{\alpha-1} t  \leq \dots \leq x^\alpha + \sum_{i=1}^{p-2} \alpha^i x^{\alpha-i} t^i + \alpha^{p-1}  (x+t)^{\alpha-p+1} t^{p-1}  \leq x^\alpha + \alpha^p t^\alpha + \alpha^p\sum_{i=1}^{p-1} x^{\alpha-i} t^i.$$
Hence, it follows that
\begin{align*} 
E\left[ \left| ((S+Q)^+)^\alpha - (S^+)^\alpha \right| \right] &= E\left[ \left( ((S+Q)^+)^\alpha - (S^+)^\alpha \right) \Indicator(Q \geq 0) \right] + E\left[ \left( S^\alpha- (S+Q)^\alpha \right) \Indicator(Q < 0 \leq S+Q) \right]  \\
&\hspace{5mm} +  E\left[ (S^+)^\alpha \Indicator(Q < 0, S+Q < 0) \right] \\
&\leq E\left[ \left( (S^+ +Q^+)^\alpha - (S^+)^\alpha \right) \Indicator(Q \geq 0) \right] + E\left[ \left( S^\alpha- (S- Q^-)^\alpha \right) \Indicator(Q < 0 \leq S+Q) \right]  \\
&\hspace{5mm} +  E\left[ ((-Q)^+)^\alpha \Indicator(Q < 0, S+Q < 0) \right] \\
&\leq E\left[ \left( \alpha^p (Q^+)^\alpha + \alpha^p \sum_{i=1}^{p-1} (S^+)^{\alpha-i} (Q^+)^i \right) \Indicator(Q \geq 0) \right]  \\
&\hspace{5mm} + E\left[ \alpha S^{\alpha-1} (Q^-) \Indicator(Q < 0 \leq S+Q) \right] + E\left[ (Q^-)^\alpha \Indicator(Q < 0, S+Q < 0) \right] \\ 
&\leq \alpha^p E[|Q|^\alpha] + 2 \alpha^p \sum_{i=1}^{p-1} E\left[(S^+)^{\alpha-i} |Q|^i \right].
\end{align*}
To see that each of the expectations of the form $E\left[(S^+)^{\alpha-i} |Q|^i \right]$ is finite note that $S^+ \leq \sum_{i=1}^N |C_iR_i|$ and follow the same steps as in the proof of Lemma 4.8 in \cite{Jel_Olv_11}. 

To establish the result for $d(t) = t^-$ simply note that 
$$E\left[ \left| ((S+Q)^-)^\alpha - (S^-)^\alpha \right| \right] = E\left[ \left| ((-S-Q)^+)^\alpha - ((-S)^+)^\alpha \right| \right]$$
and apply the result for the positive part. Finally, for $d(t) = |t|$ we use the fact that $|x|^\beta = (x^+)^\beta + (x^-)^\beta$ for any $x \in \mathbb{R}$ to obtain
\begin{align*}
E\left[ \left| |S+Q|^\alpha - |S|^\alpha \right| \right] &=  E\left[ \left| ((S+Q)^+)^\alpha  + ((S+Q)^-)^\alpha - (S^+)^\alpha - (S^-)^\alpha  \right| \right] ,
\end{align*}
which is finite by the previous two cases $d(t) = t^+$ and $d(t) = t^-$. 
\end{proof}

\begin{lem} \label{L.NewIntegralIneq}
For any two real valued random variables $X$ and $Y$ on a common probability space, 
$$\int_0^\infty E\left[ \left| 1(X > t) - 1(Y > t) \right| \right] t^{\alpha-1} dt \leq \frac{1}{\alpha} E\left[  \left| (X^+)^\alpha - (Y^+)^\alpha \right| \right],$$
finite or infinite. 
\end{lem}

\begin{proof}
Note that
$$|1(X> t) - 1(Y > t)| = |1(X > t, Y \leq t) - 1(Y > t, X \leq t) | \leq 1(Y \leq t < X) + 1(X \leq t < Y). $$
It follows from this observation and Fubini's theorem that
\begin{align*}
&\int_0^\infty E\left[ \left| 1(X > t) - 1(Y > t) \right| \right] t^{\alpha-1} dt  \\
&= E\left[  \int_0^\infty \left| 1(X > t) - 1(Y > t) \right| t^{\alpha-1} dt  \right]  \\
&\leq E\left[ \int_0^\infty 1(Y \leq t < X) t^{\alpha-1} dt + \int_0^\infty 1(X \leq t < Y) t^{\alpha-1} dt \right] \\
&= E\left[ \int_{Y^+}^{X^+} t^{\alpha-1}dt \, 1(Y^+ < X^+) + \int_{X^+}^{Y^+} t^{\alpha-1} dt \, 1(X^+ < Y^+) \right] \\
&= E\left[ \frac{1}{\alpha} ( (X^+)^\alpha - (Y^+)^\alpha) 1(Y^+ < X^+) + \frac{1}{\alpha} ((Y^+)^\alpha - (X^+)^\alpha) 1(X^+ < Y^+) \right] \\
&= \frac{1}{\alpha}  E\left[ | (X^+)^\alpha - (Y^+)^\alpha | \right].
\end{align*}
\end{proof}



\bibliographystyle{plain}

\begin{thebibliography}{10}

\bibitem{Aldo_Band_05}
D.J. Aldous and A.~Bandyopadhyay.
\newblock A survey of max-type recursive distributional equation.
\newblock {\em Annals of Applied Probability}, 15(2):1047--1110, 2005.

\bibitem{Als_Big_Mei_10}
G.~Alsmeyer, J.D. Biggins, and M.~Meiners.
\newblock The functional equation of the smoothing transform.
\newblock {\em arXiv:0906.3133}, 2010.

\bibitem{Alsm_Kuhl_07}
G.~Alsmeyer and D.~Kuhlbusch.
\newblock Double martingale structure and existence of $\phi$-moments for
  weighted branching processes.
\newblock {\em M\"unster Journal of Mathematics}, 1, 2008.

\bibitem{Alsm_Mein_10b}
G.~Alsmeyer and M.~Meiners.
\newblock Fixed points of the smoothing transform: {T}wo-sided solutions.
\newblock {\em arXiv:1009.2412}, 2010.

\bibitem{Alsm_Rosl_05}
G.~Alsmeyer and U.~R\"osler.
\newblock A stochastic fixed point equation related to weighted branching with
  deterministic weights.
\newblock {\em Electron. J. Probab.}, 11:27--56, 2005.

\bibitem{Biggins_77}
J.D. Biggins.
\newblock Martingale convergence in the branching random walk.
\newblock {\em Journal of Applied Probability}, 14(1):25--37, 1977.

\bibitem{Biggins_Kyprianou_77}
J.D. Biggins and A.E. Kyprianou.
\newblock Seneta-heyde norming in the branching random walk.
\newblock {\em Ann. Probab.}, 25(1):337--360, 1997.

\bibitem{ChowTeich1988}
Y.S. Chow and H.~Teicher.
\newblock {\em Probability Theory}.
\newblock Springer-Verlag, New York, 1988.

\bibitem{Grincevicius_75}
A.K.~Gincevi\v cius.
\newblock One limit distribution for a random walk on the line.
\newblock {\em Lithuanian Math. J.}, 15:580--589, 1975.

\bibitem{Durr_Ligg_83}
R.~Durret and T.~Liggett.
\newblock Fixed points of the smoothing transformation.
\newblock {\em Z. Wahrsch. verw. Gebeite}, 64:275--301, 1983.

\bibitem{Fill_Jan_01}
J.A. Fill and S.~Janson.
\newblock Approximating the limiting {Q}uicksort distribution.
\newblock {\em Random Structures Algorithms}, 19(3-4):376--406, 2001.

\bibitem{Goldie_91}
C.M. Goldie.
\newblock Implicit renewal theory and tails of solutions of random equations.
\newblock {\em Ann. Appl. Probab.}, 1(1):126--166, 1991.

\bibitem{Holl_Ligg_81}
R.~Holley and T.~Liggett.
\newblock Generalized potlatch and smoothing processes.
\newblock {\em Z. Wahrsch. verw. Gebeite}, 55:165--195, 1981.

\bibitem{Iksanov_04}
A.M. Iksanov.
\newblock Elementary fixed points of the {BRW} smoothing transforms with
  infinite number of summands.
\newblock {\em Stochastic Process. Appl.}, 114:27--50, 2004.

\bibitem{Jel_Olv_10}
P.R. Jelenkovi\'c and M.~Olvera-Cravioto.
\newblock Information ranking and power laws on trees.
\newblock {\em Adv. Appl. Prob.}, 42(4):1057--1093, 2010.

\bibitem{Jel_Olv_11}
P.R. Jelenkovi\'c and M.~Olvera-Cravioto.
\newblock Implicit renewal theory and power tails on trees.
\newblock {\em arXiv:1006.3295}, 2011.
\newblock To appear in Adv. Appl. Prob. 44(2).

\bibitem{Kah_Pey_76}
J.P. Kahane and J.~Peyri\`ere.
\newblock Sur certaines martingales de benoit mandelbrot.
\newblock {\em Adv. Math.}, 22:131--145, 1976.

\bibitem{Liu_98}
Q.~Liu.
\newblock Fixed points of a generalized smoothing transformation and
  applications to the branching random walk.
\newblock {\em Adv. Appl. Prob.}, 30:85--112, 1998.

\bibitem{Liu_00}
Q.~Liu.
\newblock On generalized multiplicative cascades.
\newblock {\em Stochastic Process. Appl.}, 86:263--286, 2000.

\bibitem{Negadailov_10}
P.~Negadailov.
\newblock {\em Limit theorems for random recurrences and renewal-type
  processes}.
\newblock 2010.
\newblock PhD Thesis. Available at
  http://igitur-archive.library.uu.nl/dissertations/.

\bibitem{Nei_Rus_04}
R.~Neininger and L.~R\"uschendorf.
\newblock A general limit theorem for recursive algorithms and combinatorial
  structures.
\newblock {\em Ann. Appl. Prob.}, 14(1):378--418, 2004.

\bibitem{Rosler_93}
U.~R\"{o}sler.
\newblock The weighted branching process.
\newblock {\em Dynamics of complex and irregular systems (Bielefeld, 1991)},
  pages 154--165, 1993.
\newblock Bielefeld Encounters in Mathematics and Physics VIII, World Science
  Publishing, River Edge, NJ.

\bibitem{Ros_Rus_01}
U.~R\"{o}sler and L.~R\"{u}schendorf.
\newblock The contraction method for recursive algorithms.
\newblock {\em Algorithmica}, 29(1-2):3--33, 2001.

\bibitem{Sgi_06}
M.S. Sgibnev.
\newblock The matrix analogue of the {B}lackwell renewal theorem on the real
  line.
\newblock {\em Sbornik: Mathematics}, 197(3):369--386, 2006.

\bibitem{Volk_Litv_08}
Y.~Volkovich and N.~Litvak.
\newblock Asymptotic analysis for personalized web search.
\newblock {\em Adv. Appl. Prob.}, 42(2):577--604, 2010.

\bibitem{Zwart_01}
B.~Zwart.
\newblock Tail asymptotics for the busy period in the {GI/G/1} queue.
\newblock {\em Math. Oper. Res.}, 26(3):485--493, 2001.

\end{thebibliography}

\end{document}